\documentclass[10pt]{amsart}

\makeatletter
\let\@wraptoccontribs\wraptoccontribs
\makeatother

\usepackage{amssymb}
\usepackage{comment}
\usepackage[all]{xy}
\usepackage{enumerate}

\usepackage{bbm}

\usepackage[dvipsnames]{xcolor}
\usepackage{graphicx}

\usepackage{amsmath}
\usepackage{amsxtra}
\usepackage{amscd}
\usepackage{amsthm}
\usepackage{amsfonts}
\usepackage{amssymb}
\usepackage{eucal}
\usepackage{verbatim}
\usepackage[all]{xy}
\usepackage{mathrsfs}
\usepackage{lineno}
\usepackage{color}
\definecolor{deepjunglegreen}{rgb}{0.0, 0.29, 0.29}
\definecolor{darkspringgreen}{rgb}{0.09, 0.45, 0.27}

\usepackage{stmaryrd}

\makeatletter
\usepackage{etex,etoolbox,needspace}
\pretocmd\section{\Needspace*{4\baselineskip}}{}{}
\makeatletter

\usepackage[pdftex,bookmarks=false,colorlinks=true,citecolor=darkspringgreen,debug=true,
  naturalnames=true,pdfnewwindow=true]{hyperref}

\DeclareMathOperator{\Gra}{{Gr}}

\usepackage{tikz-cd}

\newcommand{\rar}[1]{\stackrel{#1}{\longrightarrow}}

\newcommand{\bA}{{\mathbb A}}

\newcommand{\bG}{{\mathbb G}}

\newcommand{\bP}{{\mathbb P}}

\newcommand{\bT}{{\mathbb T}}

\newcommand{\bW}{{\mathbb W}}

\newcommand{\bZ}{{\mathbb Z}}

\newcommand{\cD}{{\mathcal D}}

\newcommand{\cH}{{\mathcal H}}

\newcommand{\cL}{{\mathcal L}}
\newcommand{\cM}{{\mathcal M}}

\newcommand{\cO}{{\mathcal O}}

\newcommand{\cS}{{\mathcal S}}

\newcommand{\cZ}{{\mathcal Z}}

\newcommand{\fg}{{\mathfrak g}}

\newcommand{\nG}{\widetilde G^e}
\newcommand{\nfG}{\widetilde G^{\flat , e}}

\newcommand{\nc}{\newcommand}

\nc\wh{\widehat}

\nc\on{\operatorname}

\nc\Gr{\on{Gr}}

\nc\Fl{\on{Fl}}

\DeclareMathOperator{\Lie}{{Lie}}

\DeclareMathOperator{\Mat}{{Mat}}

\DeclareMathOperator{\Mod}{{Mod}}

\DeclareMathOperator{\coker}{{coker}}
\DeclareMathOperator{\Sp}{{Sp}}
\DeclareMathOperator{\lsp}{\mathfrak{sp}}
\DeclareMathOperator{\Class}{{Cl}}
\DeclareMathOperator{\conn}{{Conn}}

\newcommand{\limto}{{\displaystyle\lim_{\longrightarrow}}}
\newcommand{\rightlim}{\mathop{\limto}}

\newcommand{\leftlim}{\mathop{\displaystyle\lim_{\longleftarrow}}}
\newcommand{\limfromn}{\leftlim\limits_{\raise3pt\hbox{$n$}}}
\newcommand{\limton}{\rightlim\limits_{\raise3pt\hbox{$n$}}}

\newcommand{\rightlimit}[1]{\mathop{\lim\limits_{\longrightarrow}}\limits%
                    _{\raise3pt\hbox{$\scriptstyle #1$}}}

\newcommand{\leftlimit}[1]{\mathop{\lim\limits_{\longleftarrow}}\limits%
                    _{\raise3pt\hbox{$\scriptstyle #1$}}}

\newcommand{\epi}{\twoheadrightarrow}
\newcommand{\iso}{\buildrel{\sim}\over{\longrightarrow}}
\newcommand{\mono}{\hookrightarrow}

\DeclareMathOperator{\Id}{{Id}}
\DeclareMathOperator{\Aut}{{Aut}}

\DeclareMathOperator{\End}{{End}} 
\DeclareMathOperator{\Hom}{{Hom}}

\DeclareMathOperator{\im}{{Im}} 
\DeclareMathOperator{\Mor}{{Mor}}

\DeclareMathOperator{\Spec}{{Spec}}

\DeclareMathOperator{\PGL}{{PGL}}
\DeclareMathOperator{\GL}{{GL}}
\DeclareMathOperator{\Pic}{{Pic}}

\DeclareMathOperator{\QCoh}{{QCoh}}

\DeclareMathOperator{\Ad}{{Ad}}

\makeatletter

\newcommand{\Rmnum}[1]{\expandafter\@slowromancap\romannumeral #1@}
\makeatother

\newtheorem{Th}{Theorem}
\newtheorem{pr}{Proposition}[section]
\newtheorem{lm}[pr]{Lemma}

\newtheorem{cor}[pr]{Corollary}

\theoremstyle{definition}

\newtheorem{rem}[pr]{Remark}

\numberwithin{equation}{section}

\newcommand{\Br}{\operatorname{Br}}

\newcommand{\modulo}{\operatorname{mod}}

\DeclareMathOperator{\ad}{{ad}}
\DeclareMathOperator{\proj}{{pr}}

\begin{document}

\title[Quantization of symplectic varieties in characteristic $p$]
{On the Bezrukavnikov-Kaledin quantization of symplectic varieties in characteristic $p$}




\author[E.~Bogdanova]{Ekaterina Bogdanova}
\address{National Research University ``Higher School of Economics'',  Russia}
\email{katbogd11@gmail.com}
\author[V. ~Vologodsky]{Vadim Vologodsky}
\address{National Research University ``Higher School of Economics'',  Russia}
\email{vologod@gmail.com}


\begin{abstract} 
We prove that after  inverting the Planck constant $h$ the Bezrukavnikov-Kaledin quantization $(X, \cO_h)$ of symplectic variety $X$  in characteristic $p$ with $H^2(X, \cO_X) =0$ is Morita equivalent
to a certain central reduction of the algebra of differential operators on $X$. 
\end{abstract}

\maketitle

\tableofcontents

\section{Introduction}
\subsection{Frobenius-constant quantizations.}
For the duration of this paper, let $k$ be a perfect field of characteristic $p>2$.  Given a scheme $X$ over $k$ we denote by $X'$ the Frobenius twist of $X$ and by $F: X\to X'$ the $k$-linear Frobenius morphism.
Since $F$ is a homeomorphism on the underlying topological spaces we shall identify the categories of sheaves on $X$ and $X'$.

Let $X$ be a smooth variety over $k$ equipped with a symplectic $2$-form $\omega$. 
Recall, that a quantization  $(X, \cO_h)$ of $(X, \omega)$ is a sheaf $\cO_h$ on the  Zariski site of $X$  of flat $k[[h]]$-algebras complete with respect to the $h$-adic topology together with an isomorphism of $k$-algebras
$$ \cO_h/h \iso  \cO_X$$
such that, for any two local sections $\tilde f$, $\tilde g$ of    $\cO_h$, one has
$$\{f,g\}\equiv  \frac{\tilde f  \tilde g - \tilde g  \tilde f}{h} \modulo h.$$ 
Here $f$ and $g$ stand  for the images in $\cO_X$ of  $\tilde f$ and  $\tilde g$ respectively and $\{, \}$ for the Poisson bracket $\cO_X$ induced by the symplectic structure.  
Note that if $X$ is affine then giving a quantization  $(X, \cO_h)$ of $(X, \omega)$ is equivalent to giving a quantization $\cO_h(X)$ of the Poisson algebra $\cO_X(X)$ (see {\it e.g.} \cite[Remark 1.6]{bk2}).

A feature special to characteristic $p$ is that the Poisson algebra $\cO_X$ of a symplectic variety has a large center consisting of $p$-th powers of functions. We are going to identify it with the sheaf $\cO_{X'}$ using 
the Frobenius morphism
$$F^*: \cO_{X'} \iso   \cO_{X}^p \subset \cO_{X}.$$
Given a quantization $(X, \cO_h)$ of $(X, \omega)$  we have $k$-linear homomorphisms 
 \begin{equation}\label{centralquantization}
 \cZ_h \epi \cZ_h/h \mono \cO_{X'}.
\end{equation}
from the center  $\cZ_h $ of the quantization $\cO_h$ to the Poisson center. Following \cite{bk}, a quantization is called central if the composition (\ref{centralquantization}) is surjective. 
A Frobenius-constant quantization 
of $(X, \omega)$ is a pair consisting of a quantization $(X, \cO_h)$ of the symplectic variety $X$ together with a $k[[h]]$-algebra isomorphism
\begin{equation}\label{center}
s:  \cO_{X'}[[h]] \iso \cZ_h
\end{equation}
such that, for any local section $f^p \in \cO_{X}^p = \cO_{X'} \subset \cO_{X'}[[h]]$ and a lift $\tilde f\in \cO_h$ of $f\in \cO_X$, one has  that\footnote{Recall, that for any associative algebra $A$ over a field of characteristic $p$ and every  elements $x, y \in A$, 
the element $(x+y)^p -x^p -y^p$ can be written as a homogeneous Lie polynomial in $x$ and $y$ of degree $p$. Applying this to $A=\cO_h$ we infer that $\tilde f ^p \mod h^p$ depends only on $f\in \cO_h/h$ and not on the choice  of
a lifting $\tilde f \in \cO_h$  of $f$. Also, using  that, for any $x\in A$,  
one has that $\ad _{x^p}= (\ad_x)^p: A\to A$, it follows that $\tilde f ^p $ is a central element of  $\cO_h/h^p$.}
$$s(f^p) = \tilde f^p \mod h^{p-1}.$$ 
It is clear that a quantization that admits a Frobenius-constant structure is central. 

A Frobenius-constant structure on  $(X, \cO_h)$ makes  $\cO_h$ into a sheaf of algebras over $ \cO_{X'}[[h]]$.
It is shown in \cite{bk} that $\cO_h$ is locally free of rank $p^{\dim X}$ as an $\cO_{X'}[[h]]$-module for the Zariski topology on $X'$. 

Frobenius-constant quantizations of symplectic varieties have been first introduced by Bezrukavnikov and Kaledin as a tool for proving  the categorical McKay correspondence for symplectic resolutions of singularities 
(see \cite{bk3}). Most of the foundational results have been obtained in \cite{bk}. The technique introduced in  \cite{bk3} has found some other applications in geometric representations theory (see {\it e.g.}, \cite{bf}, \cite{bl}).
A key to all these applications is {\it the Azumaya property} of  the algebra  $\cO_h(h^{-1})$ obtained from $\cO_h$ by inverting $h$: it is shown in \cite{bk} that for any Frobenius-constant 
quantization 
on  $(X, \cO_h)$,  the algebra $\cO_h(h^{-1})$  is isomorphic, that is locally, for the {\it fppf}  topology  on $X'$,   to  a  matrix algebra over $\cO_{X'}((h))$. Since the algebra $\cO_h$ has no zero divisors  $\cO_h(h^{-1})$
does not split even locally for the \'etale topology on $X'$ (unless $\dim X =0$). In \cite[Proposition 1.24]{bk} a formula for the class of this Azumaya algebra in an appropriate Brauer group was proposed. However, it has been observed in \cite{m}
that the formula in \cite{bk} is not correct as stated. 
The immediate goal of this paper is to correct it.  The technique introduced along its proof (in particular, the Basic Lemma from \S \ref{planofproof.intro}) plays an essential role in a sequel paper joint with Kubrak and Travkin
\cite{bktv}, where we prove that the category of quasi-coherent sheaves on any restricted  symplectic variety admits a canonical Frobenius-constant quantization.     
 
 \subsection{Differential operators as  a Frobenius-constant quantization.}\label{intro.dif}
   A basic example of a Frobenius-constant quantization is the following. Let $Y$ be a smooth variety over $k$, $X:= \bT^*_Y$ the cotangent bundle to $Y$ equipped with the canonical symplectic 
structure $\omega$. Denote by $D_{Y}$ the sheaf of differential operators on $Y$.  This comes with a filtration given by the order of a differential  operator. 
Applying the Rees construction to the filtered algebra $D_{Y}$ we obtain a sheaf of algebras $D_{Y, h}$ flat over $k[h]$ whose fiber over $h=1$ is $D_Y$ and whose fiber over $h=0$ 
is the symmetric algebra $S^\cdot T_Y$. 
Explicitly,  $D_{Y, h}$ is the subalgebra of $D_{Y}[h]$ generated by $h$, $\cO_Y$, and $hT_Y$.
The $p$-curvature homomorphism  
$$S^{\cdot} T_{Y'} \to  D_{Y, h}$$ 
sending a function $f\in \cO_Y$ to $f^p$ and a vector  field $\theta \in T_Y$ to $(h \theta)^p - h^{p-1}(h\theta^{[p]})$ induces an isomorphism between the algebra $S^{\cdot} T_{Y'}[h]$ and the center 
of   $D_{Y, h}$. In particular, $D_{Y, h}$ can be viewed as a quasi-coherent sheaf on $\bT_{Y'}^*$. The canonical  Frobenius-constant quantization of $(\bT^*_Y, \omega)$ is obtained from 
$D_{Y, h}$ by $h$-completion. We shall denote this canonical Frobenius-constant quantization of  $(\bT^*_Y, \omega)$ by  $(\bT^*_Y,  \cD_{Y, h})$.
\subsection{Restricted Poisson structures.}
There is a local obstruction to the existence of a central quantization of a symplectic variety $(X, \omega)$. It is observed in  \cite{bk} that if $f^p \in  \cO_{X'} \iso   \cO_{X}^p$ is in the image (\ref{centralquantization}) then
the restricted power $H_f^{[p]}$ of the Hamiltonian vector field $H_f$ is again Hamiltonian: $H_f^{[p]}= H_{f^{[p]}}$ for some $f^{[p]}\in \cO_{X}$. For example, it follows that the torus $(\bG_m \times \bG_m, \omega= 
 \frac{dx}{x}\wedge  \frac{dy}{y})$ does not admit central quantizations.
 
A Frobenius-constant structure on $(X, \cO_h)$ provides a canonical Hamiltonian for $H_f^{[p]}$.
In fact, given a Frobenius-constant quantization 
$(X, \cO_h, s)$ the formula 
\begin{equation}\label{restricted}
f^{[p]}= \frac{1}{h^{p-1}}   (\tilde f^p  -   s(f))   \mod h
\end{equation}
defines a restricted structure on the Poisson algebra $\cO_X$, that is the structure of a restricted Lie algebra on $\cO_X$ such that $(f^2)^{[p]}= 2 f^{[p]} f^p$ and  $H_f^{[p]}= H_{f^{[p]}}\footnote{Whereas the notion of restricted  Lie algebra goes back to Jacobson (1937),  the concept of restricted Poisson algebra is an invention of Bezrukavnikov and Kaledin  (\cite[Def. 1.9]{bk}). Note that using the identity $fg = \frac{1}{4}((f+g)^2 -(f-g)^2)$, one
has that $(fg)^{[p]} =  f^p g^{[p]} + f^{[p]} g^p + P(f, g)$, where $P(f, g)$ is an element of a free Poisson algebra on $f$ and $g$. In \cite{bk}, the authors construct $P(f, g)$ explicitly in any characteristic  which makes it possible to define the notion of restricted Poisson algebra  even in characteristic $2$.}$.

It is shown in \cite{bk}  that, for every symplectic variety  $(X, \omega)$,  giving  a restricted  Poisson structure on $\cO_X$ is equivalent to giving a class
$$[\eta] \in H^0_{Zar}(X, \coker(\cO_X \rar{d} \Omega^1_X))$$
 such that 
 $$d([\eta])=\omega.$$
 In one direction, if $\eta\in \Omega^1_X$, $d\eta=\omega$, then the formula 
 $$f^{[p]}= L^{p-1}_{H_f} \iota_{H_f} \eta - \iota_{H_f^{[p]}} \eta$$
 defines a restricted structure on $\cO_X$.
 In particular, if $(X, \omega)$ admits a restricted structure then $\omega$ is exact locally for Zariski topology on $X$. 
  \subsection{Classification of Frobenius-constant quantizations.}\label{classification}
 Fix a symplectic variety $X$ with a restricted Poisson structure $[\eta]$. Denote by $Q(X, [\eta])$ the set of isomorphism classes of Frobenius-constant quantizations 
$(X, \cO_h, s)$ compatible with $[\eta]$. In  \cite{bk}, Bezrukavnikov and Kaledin constructed a map of sets
\begin{equation}\label{bktheorem}
\rho: Q(X, [\eta]) \to  H^1_{et}(X',  \cO_{X'}^*/\cO_{X'}^{* p} )
 \end{equation}  
and showed that if  $H^1_{Zar}(X',  \cO_{X'}/\cO_{X'}^p )=0$ then  $\rho $ is injective and if  $H^2_{Zar}(X',  \cO_{X'}/\cO_{X'}^p )=0$ then  $\rho $ is surjective. Consequently, if both cohomology groups vanish the map $\rho$ is a bijection and there is a canonical Frobenius-constant quantization of  $(X, [\eta])$ corresponding to $0\in H^1_{et}(X',  \cO_{X'}^*/\cO_{X'}^{* p} )$.
  This quantization $(X, \cO_h, s)$ is uniquely characterized as the one that admits 
 a $\bZ/2$-equivariant structure:   an isomorphism 
$\cO_{X'}[[h]]$-algebras 
 \begin{equation}\label{Z/2Z-equivariant}
 \alpha:  \cO_{-h}^{op}\iso \cO_h
 \end{equation}
identical modulo $h$ and such that $\alpha \circ \alpha = \Id$.  
We review the construction of $\rho$ in \S \ref{SBKconstruction}.
\subsection{A central reduction of the algebra of differential operators.}\label{intro.centralred}  A class 
$$[\eta] \in H_{Zar}^{0}\left(X, \operatorname{coker}\left(\mathcal{O}_{X} \stackrel{d}{\longrightarrow} \Omega_{X}^{1}\right)\right)$$
gives rise to a certain central reduction $\mathcal{D}_{X, [\eta], h}$ of the algebra $\mathcal{D}_{X,  h}$. 
We first construct this reduction locally and then glue.
For any open subset $U$ together  with a $1$-form $\eta \in \Omega^1(U)$ representing  $[\eta]$,  consider the graph $\Gamma_{\eta}: U' \rightarrow \mathbb{T}_{U'}^*$
of $\eta \otimes 1 \in \Omega^1 (U')$.  Let $\mathcal{D}_{U, h}$ be the quantization of $\mathbb{T}^*_{U}$ defined above regarded as a locally free sheaf of modules over 
$S^{\cdot}T_{U'}[[h]]$ on $U$. Set
 \begin{equation}\label{graphrestriction}
\Gamma_{\eta_1}^{*} \mathcal{D}_{U, h} = \mathcal{D}_{U, h}/ I_{\Gamma_{\eta}} \mathcal{D}_{U, h}.
 \end{equation}
 Here $I_{\Gamma_{\eta}} \subset S^{\cdot}T_{U'}$ is sheaf of ideals defined by the closed embedding $\Gamma_{\eta}$. Note that $\Gamma_{\eta_1}^{*} \mathcal{D}_{U, h}$ is a sheaf of algebras
 over $S^{\cdot}T_{U'}/I_{\Gamma_{\eta}}  [[h]]\iso \cO_{U'}[[h]]$.

Suppose we are given two forms $\eta_1$, $\eta_2$ on $U$ representing the class $[\eta]$. Let us construct a canonical isomorphism between the algebras $\Gamma_{\eta_1}^{*} \mathcal{D}_{X, h}$ and $\Gamma_{\eta_2}^{*} \mathcal{D}_{X, h}$. Set    $\mu = \eta_1 - \eta_2$.  Define the automorphism $\phi_{\mu}$ of $\mathcal{D}_{X, h}$ by setting $\phi_{\mu}(f)=f$ and $\phi_{\mu}(h\theta)=h\theta + \iota_{\theta}\mu$, for any function $f$ and vector field $\theta$\footnote{Informally, this isomorphism is the conjugation by $e^{\frac{1}{h} \int \mu }$.}. Let $t_{\mu}$ be the translation by $\mu$ on $S^{\cdot}T_{U'}$ {\it i.e.},  an automorphism sending a vector field $\theta$ to $\theta  + \iota_{\theta}\mu$. Then using the Katz formula  \cite[\S 7.22]{katz}  (and the exactness of $\mu$)  the following diagram is commutative
\[
\begin{tikzcd}
      \mathcal{O}_{U'} & S^{\cdot}T_{U'} \arrow[l, "\Gamma_{\eta_1}^*"]{} \arrow[r, "s"]\arrow[d, "t_{\mu}"]&  \mathcal{D}_{U, h} \arrow[d, "\phi_{\mu}"]&   \\
      & S^{\cdot}T_{U'} \arrow[lu, "\Gamma_{\eta_2}^*"] \arrow[r, "s"]  & \mathcal{D}_{U, h}.
\end{tikzcd}
\]
The desired isomorphism is given by the formula
$$\Gamma_{\eta_1}^{*} \mathcal{D}_{U, h}  =\mathcal{D}_{U, h} \otimes_{S^{\cdot}T_{U'}} S^{\cdot}T_{U'} \slash \mathcal{I}_{\Gamma_{\eta_1}}\rar{\phi_{\mu}\otimes t_\mu}
\mathcal{D}_{U, h} \otimes_{S^{\cdot}T_{U'}} S^{\cdot}T_{U'} \slash \mathcal{I}_{\Gamma_{\eta_2}}=\Gamma_{\eta_2}^{*} \mathcal{D}_{U, h}. $$

Given three  $1$-forms $\eta_1$, $\eta_2$, and $\eta_3$ representing the class $[\eta]$ one has 
$$(\phi_{\eta_1 - \eta_2} \otimes t_{\eta_1 - \eta_2})    \circ (\phi_{\eta_2 - \eta_3} \otimes t_{\eta_2 - \eta_3}) = \phi_{\eta_1 - \eta_3}\otimes t_{\eta_1 - \eta_3} .$$
The sheaf of algebras $\mathcal{D}_{X, [\eta], h}$ is obtained by  gluing $ \Gamma_{\eta}^{*} \mathcal{D}_{U, h}$ along the above isomorphisms. 

The sheaf $\cD_{X, [\eta], h}$ of $\cO_{X'}[[h]]$-algebras  is locally free as a $\cO_{X'}[[h]]$-module of rank $p^{2\dim X}$. The commutative algebra $\cD_{X, [\eta], h}/h$ is isomorphic to 
the algebra of functions on the Frobenius neighborhood of the zero section $X\mono \bT^*_X$ with the Poisson structure given by the symplectic form $\omega_{can} + \proj ^* \omega$
on $\bT^*_X$. Here $\omega_{can}$ is the canonical symplectic form on the cotangent bundle, and $\proj: \bT^*_X \to X$ is the projection.
\begin{rem}\label{extension}
	The sheaf $\mathcal{D}_{X, [\eta], h}$ is the restriction 
	 of a certain canonical locally free $\cO_{X^{\prime} \times \bP^1}$-algebra over $X'\times \bP^1$ to the formal completion of $X' \times \{0\} \mono X'\times \bP^1$ (\cite{bktv} \S3.3). 
\end{rem}
 \subsection{Main result.} 
 Denote by $\Br(X'[[h]])$  the Brauer group of the formal scheme  $(X',  \cO_{X'}[[h]])$ obtained from $X' \times \Spec k[h]$ by completion along the closed subscheme cut by  the equation $h=0$.
 We have homomorphisms:
  \begin{equation}\label{pullbackbr}  
\delta \colon H^1_{et}(X',  \cO_{X'}^*/\cO_{X'}^{* p} )  \to H^2_{et}(X',  \cO_{X'}^* ) \cong \Br(X') \mono \Br(X'[[h]]).
 \end{equation}
 The first map in (\ref{pullbackbr}) is the boundary morphism associated to the short exact sequence of sheaves for the \'etale topology 
$$0 \to  \cO_{X'}^* \rar{p}  \cO_{X'}^* \rar{} \cO_{X'}^*/\cO_{X'}^{* p} \to 0.$$
The right arrow   in (\ref{pullbackbr}) is the pullback homomorphism which is a split injection because its composition with the restriction homomorphism 
$$i^*\colon  \Br(X'[[h]]) \to \Br(X')$$
 is the identity. 
Given a class $\gamma \in  H^1_{et}(X',  \cO_{X'}^*/\cO_{X'}^{* p} )$ we denote by $\delta(\gamma) \in  \Br(X'[[h]])  $  the image of $\gamma$ under
the composition (\ref{pullbackbr}).
Finally, we can state the main result of this paper.
\begin{Th}\label{main} Let $(X, \omega)$ be a smooth symplectic variety of dimension $2n$ over an algebraically closed field $k$ of characteristic $p>2$, and let 
 $(X, \cO_h, s)$ be a  Frobenius-constant  quantization 
of $(X, \omega)$.  Denote by  $[\eta] \in H^0_{Zar}(X, \coker(\cO_X \rar{d} \Omega^1_X))$ the  restricted Poisson structure corresponding to $(X, \cO_h, s)$ and by 
$\gamma= \rho(X, \cO_h, s) \in  H^1_{et}(X',  \cO_{X'}^*/\cO_{X'}^{* p} )$ the image of $(X, \cO_h, s)$ under   (\ref{bktheorem}).
Then there exists an Azumaya algebra $\cO^\sharp_h$ over the formal scheme  $(X',  \cO_{X'}[[h]])$  with the following properties:
\begin{itemize}
\item[(i)]
There exists an isomorphism of $\cO_{X'}((h))$-algebras
 \begin{equation}\label{maintheq1}  
(\cO_h \otimes_{\cO_{X'}[[h]]} \cD_{X, [\eta], h}^{op})(h^{-1})\iso \cO^\sharp_h(h^{-1})
\end{equation}
\item[(ii)] We have that
 \begin{equation}\label{formularestrintro}
 i^*[\cO^\sharp_h]= i^*(\delta(\gamma)).
  \end{equation}
In partucular, if $H^2(X, \cO_X) =0$ then $[\cO^\sharp_h] =  \delta(\gamma)$.
\end{itemize}
\end{Th} 
In particular, let  $(X, \cO_h, s)$ be a  Frobenius-constant  quantization that admits $\bZ/2\bZ$-equivariant structure  (\ref{Z/2Z-equivariant}). Assume that $H^2(X, \cO_X) =0$. Then by Theorem \ref{main} 
$\cO^\sharp_h$ is a split Azumaya algebra, that is there exists a locally free $\cO_{X'}[[h]]$-module $E$ of finite rank and an isomorphism of  $\cO_{X'}[[h]]$-algebras
$$\cO^\sharp_h \iso \End_{\cO_{X'}[[h]]}(E).$$
Using (\ref{maintheq1}) and the Azumaya property of   $\cO_h (h^{-1})$ and $\cD_{X, [\eta], h}(h^{-1})$ we observe an equivalence of categories 
\begin{equation}\label{modules}
	 \Mod(\cD_{X, [\eta], h}(h^{-1}))\iso  \Mod(\cO_h (h^{-1})) 
\end{equation}
between the category of $\cO_h (h^{-1})$-modules and  the category of $\cD_{X, [\eta], h}(h^{-1})$-modules.
The functor from left to right carries a $\cD_{X, [\eta], h}(h^{-1})$-module $M$ to $E\otimes _{\cD_{X, [\eta], h}} M$; the quasi-inverse functor takes an $\cO_h (h^{-1})$-module $N$ to 
${\cH}om_{\cO_h }(E, N)$.

Also note that if $H^1_{Zar}(X',  \cO_{X'}/\cO_{X'}^p )=0$ then the map 
$$H^0(X, \Omega^1_X) \to  H^0_{Zar}(X, \coker(\cO_X \rar{d} \Omega^1_X))$$ 
is surjective and, thus, any  restricted Poisson structure arises from a global $1$-form $\eta$. In this case  objects of  $\Mod(\cD_{X, [\eta], h}(h^{-1}))$ can be viewed as $\cD_{X, h}(h^{-1})$-modules whose $p$-curvature
equals $\eta$.

\subsection{$\bG_m$-equivariant quantizations.}

Let  $(X, \omega)$ be a symplectic variety  equipped with an action 
\begin{equation}\label{multgractionintro}
\lambda \colon \bG_m \times X \to X
\end{equation}
of the multiplicative group   such that $\omega$ has a
positive weight $m$ with respect to this action. Moreover, we shall assume that $m$ invertible in $k$. Denote by $\theta$
the Euler vector field on $X$ corresponding to the $\bG_m$-action. Then the formula
$\eta = \frac{1}{m} \iota_{\theta} \omega $ defines a restricted structure on $X$. Define a $\bG_m$-action on  $X'$  twisting  (\ref{multgractionintro}) by the $p$-th power map $\bG_m \rar{F} \bG_m$.
Also let $\bG_m$ act on $X'[h]:= X' \times \Spec k[h]$ as above on the first factor and  by $z * h = z^m h$ on the second one.  

A $\bG_m$-equivariant Frobenius-constant quantization  of $X$ is a $\bG_m$-equivariant sheaf $O_h$ of associative $\cO_{X'[h]}$-algebras on $X'[h]$, locally free as an $\cO_{X'[h]}$-module,
such that the restriction $\cO_h$ 
of $O_h$ to the formal completion of $X'[h]$ along the divisor $h=0$ is a Frobenius-constant quantization  of  $X$
compatible with the restricted structure  $[\eta]$.  Examples of $\bG_m$-equivariant  quantizations arise in geometric representation theory  (see {\it e.g.}, \cite{bk3}, \cite{bf}, \cite{bl}, \cite{kt}).

Assume that morphism  (\ref{multgractionintro}) extends to a morphism 
\begin{equation}\label{multgractionextintro}
\tilde \lambda: \bA^1 \times X \to X.
\end{equation}
Then the restriction of $O_h$ to  the open subscheme  $X'[h, h^{-1}]\mono X'[h]$ is an Azumaya algebra.  As an application of Theorem \ref{main} we prove in \S \ref{ktconjecture}
a conjecture of Kubrak-Travkin concerning the class  of this algebra in the Brauer group. 
Namely, we show that, for every $\bG_m$-equivariant Frobenius-constant quantization $O_h$, 
the following equality in $\Br(X')$ holds
$$[O_{h=1}] = [\frac{1}{m} \eta]  +  \tilde \lambda_0^*[\rho(O_h)] .$$
Here  $[\eta]$ stands for the image of $\eta$ under the canonical map $\Gamma (X', \Omega^1_{X'}) \to \Br(X')$,  $\rho(O_h)\in  H^1_{et}(X',  \cO_{X'}^*/\cO_{X'}^{* p} )$ for the class associated to the formal quantization 
via   (\ref{bktheorem}), and $[\rho(O_h)]$ for its image in the Brauer group.

 \subsection{Plan of the proof.}\label{planofproof.intro}
 Using the language of formal geometry we reduce the theorem to a group-theoretic statement. We shall start by explaining the latter. 
 
 Let $(V, \omega_V)$ be a finite-dimensional symplectic vector space over $k$, and let  $A_h$ be the algebra over $k[[h]]$ generated by the dual vector space $V^*$ subject to the relations 
 \[
 f g - g f =  \omega_V^{-1}(f, g) h, \quad f^p = 0
 \]
 for any $f, g \in V^*$. We refer to $A_h$ as the restricted {\it Weyl algebra}. This is a flat $k[[h]]$-algebra whose reduction modulo $h$ is the finite-dimensional commutative algebra of functions on the Frobenius neighborhood of the origin in the affine space $\Spec(S^{\cdot} V^*) := \bf{V}$. Explicitly, $A_0 := S^{\cdot} V^* / J_V$ where $J_V$ is the ideal generated by $f^p$ for all $f \in V^*$. The quantization $A_h$ of $A_0$ specifies a restricted Poisson structure $[\eta_V]\in \coker(A_0 \rar{d} \Omega^1_{A_0})$.\footnote{Explicitly $[\eta_V]$ is characterized as a unique homogeneous class such that $d[\eta_V] = \omega_V$.}
 
  Denote by $G$ the group scheme $\underline{\Aut}(A_h)$ of $k[[h]]$-linear automorphisms of the algebra $A_h$, by $G^{\geq 1}$ the subgroup of automorphisms identical modulo $h$, and by $G_0$ the quotient of $G$ by $G^{\geq 1}$. As shown in \cite{bk}, $G_0$ is the group scheme of automorphisms of $A_0$ preserving the class $[\eta_V]$.
 
A pair $(W, W^*)$ of transversal Lagrangian subspaces of $V$ defines an isomorphism between the algebra $A_h(h^{-1})$ and the matrix algebra $\End_{k}(S^{\cdot} W^* / J_W)((h))$ which, in turn, gives an embedding $G \mono L\PGL(p^n)$, where $L\PGL(p^n)$ is the loop group of $\PGL(p^n)$ (viewed as a sheaf for the {\it fpqc} topology). Then the extension $$1 \to \mathbb{G}_m \to \GL(p^n) \to \PGL(p^n) \to 1$$ gives rise to\footnote{We do not know if the morphism of {\it fpqc}  sheaves  $L\GL(p^n)\to  L\PGL(p^n)$  is surjective. However, we check in Proposition \ref{apploopspgl} that its pullback to any group subscheme $G\subset L\PGL(p^n)$ satisfying some finiteness assumptions is surjective even for the Zariski topology on $G$.} 
 \begin{equation}\label{G}
 1 \to L\bG_m \to \tilde G  \to   G \to 1. 
 \end{equation} In \cite{bk} it is proved that $G^{\geq 1}$ is the subgroup of inner automorphisms. Hence we have a subextension of \eqref{G}
 \begin{equation}
 1 \to L^+\bG_m \to \underline{A_h^*} \to   G^{\geq 1} \to 1,
 \end{equation} 
 where $L^+\bG_m$ is the positive loop group of $\mathbb{G}_m$.
 Then passing to the quotient we get a central extension by the affine grassmannian\footnote{Recall from \cite{cc} that $\Gra_{\bG_m}$ is isomorphic to the direct product $\hat \bW \times \underline{\bZ}$, where $\hat \bW$
 a group ind-scheme whose points with values in a $k$-algebra $R$ is the subgroup of  $R[h^{-1}]^*$ consisting of invertible polynomials with zero constant term. In particular, if $R$ is reduced and connected then 
 $\Gra_{\bG_m}(R)=\bZ$.}  
 \begin{equation}\label{ext}
 1 \to \Gra_{\bG_m} \to \tilde G_0  \to   G_0 \to 1. 
 \end{equation}
 
 Let $i: V\mono V^{\flat}$ be a morphism of symplectic vector spaces such that the restriction to $V$ of the symplectic form on $V^{\flat}$ is $\omega_V$. Let $\widetilde G_0 \to G_0$ and $\widetilde{G}_0^{\flat} \to G_0^{\flat}$ be the corresponding extensions. We emphasize that $\widetilde G_0$ and $\widetilde G_0^{\flat}$ depend on a choice of Lagrangian pairs in $V$ and $V^{\flat}$.
 
 Finally, denote by $G_0^{\sharp}\subset G_0^\flat $ the group subscheme that consists of automorphisms preserving the kernel of the homomorphism $i^*: A_0^{\flat} \to A_0$. 
 We have a natural homomorphism $G_0^{\sharp} \to G_0$.
 In \S \ref{Skeyproposition} we prove the following assertion:  

 	{\bf Basic lemma.} {\it The homomorphism $G_0^{\sharp} \to G_0$ lifts uniquely to a homomorphism of central extensions
 	$$ \widetilde G_0^\flat  \times _{G_0^\flat}  G_0^{\sharp} \to \widetilde G_0.$$}

 Our proof of the Basic Lemma, that occupies almost the half of the paper, is based on a new construction of  (\ref{ext})  that makes this functoriality property obvious. Namely, consider two subgroups $\alpha \subset G_0 \supset G_0^0$, where   $G_0^0$ is the subgroup of automorphisms preserving the origin in $\bf{V}$ (which by a result of Bezrukavnikov and Kaledin coincides with the reduced subgroup of $G_0$) and 
 $\alpha=\Spec A_0 (\cong \alpha_p^{\dim V})$ 
 is the finite group scheme of translations. Then the product map $ \alpha \times G_0^0 \to G_0$ induces an isomorphism of the underlying schemes.  Let $ \widetilde \alpha $ be the restriction of  the central extension  (\ref{ext})
 to $\alpha$. This is a version of the Heisenberg group. 
We show that the extension  (\ref{ext}) splits uniquely over 
 the reduced subgroup $G_0^0$ \footnote{ A posteriori, this is a corollary of the Basic Lemma applied to the embedding $0 \mono V$.}. Thus we can view $G_0^0$ as a subgroup of $\widetilde{G_0}$, and the quotient $\widetilde{G_0}/G_0^0$  is  identified (as an ind-scheme)  with $\widetilde{\alpha}$. 
 The left action of $\widetilde{G_0}$ on  $\widetilde{G_0}/G_0^0$ defines an embedding of $\widetilde{G_0}$ into the group  of automorphisms of $\widetilde{\alpha}$
 viewed as a space with an action of $\Gra_{\bG_m}$. We prove in Theorem \ref{widetildeG_0} that the image of this embedding  is precisely the group of automorphisms that preserve a unique $\Sp_{2 n} \ltimes \alpha_{p}^{2 n}$-invariant connection
 on $\Gra_{\bG_m}$-torsor $\widetilde \alpha$.
 
 To derive the basic lemma from the above we classify all central extensions of  $\alpha $ by $ \Gra_{\mathbb{G}_m}  $ in \S \ref{alphaext}. In particular, we show that extensions of $\alpha$ by $ \Gra_{\mathbb{G}_m}  $ that split over every $\alpha_{p}$ factor are classified by $\Lie(\Gra_{\mathbb{G}_m})$-valued skew-symmetric $2$-forms on $\Lie(\alpha)$. 
 Then it follows that the morphism $\alpha \rightarrow \alpha^{\flat}$ induced by $i$ lifts uniquely to a morphism of extensions $\widetilde{\alpha} \rightarrow \widetilde{\alpha}^{\flat}$ respecting the connections. Since $\widetilde G_0^\flat  \times _{G_0^\flat}  G_0^{\sharp}$ is the group of automorphisms of $\widetilde{\alpha}^{\flat}$ that preserve the connection and the subspace $\widetilde{\alpha} \mono \widetilde{\alpha}^{\flat}$, by restriction we get the desired lifting $\widetilde G_0^\flat  \times _{G_0^\flat}  G_0^{\sharp} \to \tilde G_0.$

Let us explain how the Basic Lemma implies the theorem. The Bezrukavnikov-Kaledin construction of Frobenius-constant quantizations is based on a characteristic $p$ version of the Gelfand-Kazhdan formal geometry. Namely, it is shown in \cite{bk3} that any Frobenius-constant quantization is locally for the \textit{fpqc} topology on $X^{\prime}$ isomorphic to the constant quantization $\mathcal{O}_{X^{\prime}}[[h]] \otimes_{k[[h]]} A_h$ for a fixed finite-dimensional space $V$ of dimension $2n = \dim X$. It follows that a Frobenius-constant quantization $(X, \mathcal{O}_h, s)$ gives rise to a torsor $\mathcal{M}_{X, \mathcal{O}_h, s}$ over $G$. Conversely, the algebra $\mathcal{O}_h$ is the twist of  $\mathcal{O}_{X^{\prime}}[[h]] \otimes_{k[[h]]} A_h$ by the torsor $\cM_{X, \mathcal{O}_h, s}$, i.e.
$$\cM_{X, \mathcal{O}_h, s} {\times} ^{G} (\mathcal{O}_{X^{\prime}}[[h]] \otimes_{k[[h]]} A_h) \iso \mathcal{O}_h.$$

The reduction of differential operators $\cD_{X, [\eta],  h}$ also can be constructed using formal geometry.
Namely, choosing a homogeneous form $\eta_V$ in the class $[\eta_V]$ on $V$ consider its graph 
$$\bf{V} \mono \mathbb{T}^*_{\bf{V}},$$
and let $i : V \to V\oplus V^*= V^{\flat}$ be the corresponding linear map of vector spaces. Let $G_0^{\sharp, f}$ be the subgroup of $G_0^{\sharp} \subset G_0^{\flat}$ of automorphisms $g$ of $\alpha^{\flat}$ respecting the fibers of the projection $\pi : \alpha^{\flat} \to \alpha$, that is fitting in

	\[
\begin{tikzcd}
\alpha^{\flat} \arrow[r, "g"]\arrow[d, "\pi"]& \alpha^{\flat} \arrow[d, "\pi"]  \\
 \alpha \arrow[r, "\bar{g}"] &  \alpha.
\end{tikzcd}
\]
Then the restriction of the natural map $G_0^{\sharp} \to G_0$ to $G_0^{\sharp, f}$ is an isomorphism. This yields a homomorphism $\psi_0 : G_0 \mono G_0^{\sharp} \subset G_0^{\flat}$. In \S \ref{Dformal} we construct a lifting $\psi : G \to G^{\flat}$ of $\psi_0$ that makes $\mathcal{D}_{X, [\eta], h}$ a twist of $\mathcal{O}_{X^{\prime}}[[h]] \otimes_{k[[h]]} A_h^{\flat}$ by $\cM_{X, \mathcal{O}_h, s}$. 

Then we consider the diagram 
\[
\begin{tikzcd}
& & \Gra_{\bG_m} \arrow[d, ""]\\
 & & L\GL(p^{3n}) / L^+ \mathbb{G}_m \arrow[d, ""]  \\
G \arrow[rru, dashrightarrow,  ""]\arrow[r, ""]& \underline{\Aut}(A_h \otimes A_h^{\flat, op}) \arrow[r, ""] &  L\PGL(p^{3n}).
\end{tikzcd}
\]

Here the loop group $L\PGL$ is the group ind-scheme of projective automorphisms of a certain vector space $U$ over $k((h))$. Suppose we can construct a $G$-invariant $k[[h]]$-lattice $\Lambda$ in $U$. Then
 $$\cO^\sharp_h: = \cM_{X, \cO_h, s} \times ^G  (\mathcal{O}_{X^{\prime}}[[h]] \otimes_{k[[h]]} \End_{k[[h]]}(\Lambda))$$
 does the job for the first part of the theorem. By a general result proven in Appendix (Proposition \ref{apploopspgl}) the existence of an invariant lattice is equivalent to the existence of the dashed arrow making the diagram above commutative. This is where we use the Basic Lemma. It follows from the latter that 
 	$$ \widetilde G_0^\flat  \times _{G_0^\flat}  G_0^{\sharp, f} \cong \widetilde G_0.$$
 We infer (Proposition \ref{keyprop}) that $\psi_0$ lifts to $\widetilde \psi_0 : \widetilde G_0 \to \widetilde G_0^{\flat}$, which implies the existence of the dashed arrow.

The proof of the second part of the theorem  amounts to unveiling the Bezrukavnikov-Kaledin construction of the map $\rho$.

\subsection{Further directions.}\label{further.directions.intro}
In this subsection we briefly discuss some applications of Theorem \ref{main} obtained in a sequel paper joint with Dmitry Kubrak and Roman Travkin \cite{bktv}.

According to the Bezrukavnikov-Kaledin theorem from  \S\ref{classification} every  smooth affine restricted symplectic  variety admits a unique up to a non-canonical  isomorphism Frobenius-constant  quantization $O_h$ with 
$\rho[O_h]=0$. The formation of $O_h$ is not functorial in $(X, [\eta])$. However, we show in \cite{bktv} that the assignment $(X, [\eta]) \mapsto \Mod(O_h(X)) $ extends to a contravariant functor from the category of smooth affine restricted symplectic varieties and open embeddings to the category of abelian categories. Applying the right Kan extension this yields a functorial quantization $\QCoh_h$ of the category of quasi-coherent sheaves of any smooth restricted symplectic variety.
Moreover, using Remark \ref{extension} and equivalence  (\ref{modules}) we extend the range of  quantum parameter $h$ from being a formal variable to a genuine coordinate on $\bP^1$. The construction of $\QCoh_h$ uses in an essential way Corollary  \ref{refforbktv} of the Basic Lemma. 

Let $Y\mono X$ be a smooth Lagrangian subvariety such that  $[\eta]_{|Y} = 0$ in 
$H_{Zar}^{0}\left(Y,  \Omega_{Y}^{1}/d\cO_Y\right)$. Using results from \cite{Mu}  we show in  \S6.1 of  \cite{bktv} that every such $Y$ determines a canonical object in $\QCoh_h$, which is a quantization of the line bundle $(\Omega_{Y}^n)^{\frac{1-p}{2}}$ viewed as a quasi-coherent sheaf on $X$.

\subsection{Plan of the paper.}   In \S \ref{SBKconstruction} we review the Bezrukavnikov-Kaledin construction of Frobenius-constant quantizations which is based on a characteristic $p$ version of the Gelfand-Kazhdan formal geometry. In \S \ref{Sreduction} we recast the construction of $\cD_{X,  h}$ using the language of formal geometry and reduce Theorem \ref{main}  to a certain statement on central
extensions of the group of automorphisms of the restricted Weyl algebra. In \S \ref{Skeyproposition}  we prove this statement. In \S \ref{ktconjecture} we study $\mathbb{G}_m$-equivariant quantizations and prove a conjecture of Kubrak and Travkin. Finally, in the Appendix we prove some results (probably known to experts) on loop groups that are used in the main body of the paper.

\subsection{Acknowledgements.} The paper has grown  out of our attempt to correct an error in \cite[Proposition 1.24]{bk}. This proposition asserts, in particular,  that the algebra $\cO_h(h^{-1})$ extends to an Azumaya algebra over the formal scheme $X'[[h]]$.
In fact, the result stated in {\it loc. cit.} is similar to our formula (\ref{formularestrintro}) with the exception that the class $[\cO^\sharp_h]$
  at the left-hand side of  (\ref{formularestrintro}) is replaced by the class of an extension of $\cO_h(h^{-1})$ to $X'[[h]]$. In particular, \cite[Proposition 1.24]{bk} asserts that the Azumaya algebra 
$\cO_h(h^{-1})$  corresponding to  the $\bZ/2\bZ$-equivariant Frobenius-constant  quantization  splits which is definitely not the case since  the algebra $\cO_h$  has no zero divisors. 
In \S \ref{Sreduction}, Remark \ref{errorinbk},  we indicate
where the error in the proof of \cite[Proposition 1.24]{bk} is. In particular, we will see that $\cO_h(h^{-1})$ never extends to an Azumaya algebra over  $X'[[h]]$.

The authors would like to express their gratitude to Dmitry Kaledin.  Numerous conversations with Dmitry and his encouragements 
helped us to overcome many of the technical and conceptual difficulties we encountered in the course of the work. The authors also benefited greatly from conversations with Michael Finkelberg and Alexander Petrov on
the affine grassmannian and rigid analytic geometry respectively. We are grateful to Roman Travkin who explained to us that one should look at quantizations of the category $\QCoh(X)$ rather than  
quantizations of the algebra  $\cO_X$. This has led to  a revision of this paper as well as the appearance of a sequel paper \cite{bktv}  joint with Dmitry Kubrak and  Roman Travkin. 
Finally, we would like to thank the referee for a meticulous review.

Both authors were partially supported by the Russian Science Foundation, grant
N\textsuperscript{\underline{o}}~$21-11-00153$.

\section{Review of the Bezrukavnikov-Kaledin construction}\label{SBKconstruction}
For reader's convenience we  review the Bezrukavnikov-Kaledin construction of quantizations. We also introduce some notations to be used later. Nothing in this section is an invention of the authors.
\subsection{Darboux Lemma in characteristic $p$.} 
Our proof  of Theorem \ref{main}, as well as the Bezrukavnikov-Kaledin  construction of quantizations, is based on a version of the Gelfand-Kazhdan formal geometry that makes it possible to localize the problem and ultimately reduce it to a statement in group theory. 
The main idea is as follows.  For a symplectic variety $X$, the Poisson  bracket on $\cO_X$ is $\cO_{X'}$-linear. Therefore, we can view  $X$  as a Poisson scheme over $X'$.  For any restricted structure on $\cO_X$, one has $\cO_{X'}^{[p]}=0$. Therefore, a symplectic variety $X$ with a restricted structure can be viewed as a restricted  Poisson scheme over $X'$.
Consider the constant restricted  Poisson scheme over $X'$:
\begin{equation}\label{standard}
X' \times \Spec A_0 \to X',
\end{equation}
where 
$$A_0= k[x_1, y_1, \cdots, x_{n}, y_n]/(x_1^p, y_1^p, \cdots, x_{n}^p, y_n^p)$$
$2n=\dim X$, the morphism (\ref{standard}) is the projection to the first factor, the Poisson structure is given by symplectic  form $\sum_i  dy_i\wedge dx_i$, and the restricted structure is determined by $x_i^{[p]} = y_i^{[p]}=0$. A key insight of Bezrukavnikov and Kaledin is that any smooth symplectic variety $X$ with a restricted structure, viewed as   a restricted  Poisson scheme over $X'$, is
locally for the {\it fpqc} topology on $X'$ isomorphic to the constant restricted  Poisson scheme  $X' \times \Spec A_0 \to X'$. This is an analogue of the Darboux Lemma. 
\subsection{Quantum Darboux Lemma.} 
There is also  a quantum version of the Darboux Lemma proven in \cite{bk}: for any Frobenius-constant quantization $(X, \cO_h, s)$, the sheaf of associative $\cO_{X'}[[h]]$-algebras $\cO_h$ is isomorphic locally for the {\it fpqc} topology on $X'$ to the $h$-completed tensor product $\cO_{X'}\otimes _k A_h$, where $A_h$ is {\it the reduced Weyl algebra} that is  the $k[[h]]$-algebra generated by variable $x_i, y_i$, $(1\leq i,j, \leq n)$, subject to the relations
\begin{equation}\label{weil}
 y_j x_i - x_i y_j =\delta_{ij} h, \quad  x_i^{p} = y_i^{p}=0.
\end{equation} 
\subsection{Formal Geometry.}\label{section.fg}
Let $\underline{\Aut}(A_0)$ be the group scheme of automorphisms of the algebra $A_0$.  For any smooth scheme $X$ over $k$ of dimension $2n$, assigning to a scheme $Z$ over $X'$ the set 
$\cM_X(Z)$ of isomorphisms
$$ Z\times \Spec A_0 \iso Z \times _{X'} X$$
of schemes over $Z$, we get a $\underline{\Aut}(A_0)$-torsor over $X'$. Next, let $G_0\subset  \underline{\Aut}(A_0)$ be the group subscheme consisting of automorphisms of $A_0$ that preserve the restricted Poisson
structure on $A_0$. Then the   Darboux Lemma above implies that, for every  symplectic  variety $(X, [\eta])$ with a restricted structure  of dimension $2n$ the functor assigning
to a scheme $Z$ over $X'$ the set 
$\cM_{X, [\eta]}(Z)$ of isomorphisms
$$ Z\times \Spec A_0 \iso Z \times _{X'} X$$
of restricted Poisson schemes over $Z$ is a $G_0$-torsor over $X'$. Using the faithfully flat descent one gets a bijection  between the set of nondegenerate (that is arising from a symplectic form) restricted Poisson structures  $[\eta]$  on $X$ and the set of $G_0$-torsors over $X'$ equipped with an isomorphism of $\Aut(A_0)$-torsors
$$\underline{\Aut}(A_0)\times ^{G_0}  \cM_{X, [\eta]} \iso  \cM_X.$$
Lastly, the set of all Frobenius-constant quantizations $(X, \cO_h, s)$ of $X$ such that the induced Poisson structure on $X$ is nondegenerate is in bijection with the set of
torsors   $\cM_{X, \cO_h, s}$ over  the group scheme $G: = \underline{\Aut}(A_h)$ of automorphisms of $k[[h]]$-algebra $A_h$ (that is a group scheme whose group of points with values in a $k$-algebra $R$ is the group
of $R[[h]]$-algebra automorphisms of the $h$-adically completed tensor product    $A_h \hat{\otimes} R$)
together with an isomorphism 
$\underline{\Aut}(A_0)$-torsors
\begin{equation}\label{G_0comp}
\underline{\Aut}(A_0)\times ^{G}  \cM_{X, \cO_h, s} \iso  \cM_X.
\end{equation} 
In particular,  for a symplectic variety with a restricted structure $(X, [\eta])$, giving a Frobenius constant quantization of  $(X, [\eta])$ is equivalent to lifting a 
$G_0$-torsor $\cM_{X, [\eta]} $ to a $G$-torsor $\cM_{X, \cO_h, s} $ along the group scheme homomorphism 
\begin{equation}\label{agrphom}
G  \to G_0.
\end{equation} 
\subsection{Automorphisms of the  reduced Weyl algebra.}\label{subsectionautomorphisms}
It is shown in \cite{bk} that homomorphism (\ref{agrphom}) is surjective and its kernel $G^{\geq 1}$ consists of inner automorphisms. We have the exact sequence
\begin{equation}\label{fes}
 1\to L^+\bG_m   \to \underline {A}_h ^* \to G \to  G_0 \to 1.
\end{equation} 
Here $ \underline A_h ^*$  (resp. $L^+\bG_m $ ) is the group scheme over $k$ whose group of $R$-points is $(R\otimes A_h)^*$  (resp. $R[[h]]^*$).
Letting $G^{\geq n} \subset  G$,  ($n\geq 0$), be the group  subscheme consisting of automorphisms that are identical modulo $h^{n}$, we have that 
$ G^{\geq n}/G^{\geq n+1}\iso  \underline {A}_0/\bG_a$ for every $n>1$ and $ G^{\geq 1}/G^{\geq 2}\iso \underline {A}_0^*/\bG_m$,  $ G^{\geq 0}/G^{\geq 1}\iso G_0$.

Consider the isomorphism of $k[[h]]$-algebras 
\begin{equation}\label{Z/2Z-equivariantstronA_h}
 \alpha:  A_{-h}^{op}\iso A_h
 \end{equation}
sending    $x_i$ to $x_i$ and $y_j$ to $y_j$.  The conjugation by $\alpha$ defines an involution $\tau: G\to G$ preserving the subgroups $G^{\geq n}$, ($n\geq 0$), such that the induced action on 
$ G^{\geq n}/G^{\geq n+1}$ takes an element $g$ to $g^{(-1)^n}$.
 In particular, it follows that the extension 
$$         1\to  G^{\geq 1}/G^{\geq 2}\to           G /G^{\geq 2} \to G_0 \to 1$$
has a unique $\bZ/2\bZ$-equivariant splitting   $$G /G^{\geq 2}  \iso G_0 \ltimes \underline{A}_0^*/\bG_m. $$

\subsection{$\bZ/2\bZ$-equivariant structures.}
 Any $\bZ/2\bZ$-equivariant  Frobenius-constant quantization  $(X, \cO_h, s, \alpha)$ is isomorphic locally for  the {\it fpqc} topology on $X'$ to the $h$-completed tensor product $\cO_{X'}\hat \otimes _k A_h$ equipped with the
 equivariant structure (\ref{Z/2Z-equivariantstronA_h}). Indeed, consider   the action of $\bZ/2\bZ$ on $G^{\geq 1}$ given by $\tau$. Then $H^1(\bZ/2\bZ, G^{\geq 1})=0$ as $ G^{\geq 1}$ has a filtration $G^{\geq n}$ with 
 uniquely $2$-divisible quotients. It follows that every two $\bZ/2\bZ$-equivariant structures on $\cO_{X'}\hat \otimes _k A_h$  are locally isomorphic. Now the claim follows from the quantum  Darboux Lemma.

Consequently, giving  a $\bZ/2\bZ$-equivariant  Frobenius constant quantization of  $(X, [\eta])$ is equivalent to lifting a 
$G_0$-torsor $\cM_{X, [\eta]} $ to a torsor $\cM_{X, \cO_h, s, \alpha} $ over the subgroup $G^{\tau =1} \subset G$ of $\tau$-invariants.
\subsection{Construction of quantizations.}
Bezrukavnikov and Kaledin use the standard obstruction theory to classify liftings of a given $G_0$-torsor to a $G$-torsor. Namely, 
given a $G_0$-torsor  $\cM_{X, [\eta]}$ the set of isomorphism classes of its liftings  to a torsor
over $G/G^{\geq 2}\iso  G_0 \ltimes \underline{A}_0^*/\bG_m $ is identified with  the set of isomorphism classes of torsors over  the smooth group scheme $\underline{A}_0^*/\bG_m  \times _{G_0}  \cM_{X, [\eta]}$ over $X'$. 
The latter group scheme is identified with the quotient of the group scheme of invertible elements in the sheaf of $\cO_{X'}$-algebras $F_*\cO_X$ by constant group scheme $\bG_m \times X'$.   
Using
smoothness of $\underline{A}_0^*/\bG_m  \times _{G_0}  \cM_{X, [\eta]}$ every torsor over this group scheme is locally trivial for the \'etale topology on $X'$. Hence,  the set of isomorphism classes of
$\underline{A}_0^*/\bG_m  \times _{G_0}  \cM_{X, [\eta]}$-torsors is in bijection with 
$H^1_{et}(X', \cO_{X'}^*/\cO_{X'}^{p *}) $.
This defines a map of sets
\begin{equation}\label{bktheoremexpl}
\rho: Q(X, [\eta]) \to  H^1_{et}(X',  \cO_{X'}^*/\cO_{X'}^{* p} )
 \end{equation}  
 from the set  $Q(X, [\eta])$  of isomorphisms classes of Frobenius-constant quantizations 
$(X, \cO_h, s)$ compatible with $[\eta]$ to the \'etale cohomology group classifying torsors over $G_0 \ltimes \underline{A}_0^*/\bG_m $  lifting 
the $G_0$-torsor   $\cM_{X, [\eta]}$.  Note that under this identification  the trivial cohomology class corresponds to the lifting obtained from $\cM_{X, [\eta]}$ via the natural group homomorphism
$ G_0 \to  G_0 \ltimes \underline{A}_0^*/\bG_m$.

Next,  the obstruction class to lifting of a
$G/G^{\geq n}$-torsor, with $n>1$, to  a $G/G^{\geq n+1}$-torsor lies in $H^2(X', \cO_{X'} /\cO_{X'}^p)$. If the obstruction class vanishes then the set  of isomorphism classes of the liftings is a torsor over  
$H^1(X', \cO_{X'} /\cO_{X'}^p)$.
Hence, if  $H^1_{Zar}(X',  \cO_{X'}/\cO_{X'}^p )=0$ then  $\rho $ is injective and if  $H^2_{Zar}(X',  \cO_{X'}/\cO_{X'}^p )=0$ then  $\rho $ is surjective. In particular, if the two cohomology groups vanish 
$\rho$ is a bijection.
 The trivial cohomology class corresponds to a quantization that admits (a unique) $\bZ/2\bZ$-equivariant structure.

\section{Reduction of the main theorem to a group-theoretic statement}\label{Sreduction}
 In this section we recast the construction of $\cD_{X, [\eta], h}$ using the language of formal geometry, 
 and reduce Theorem \ref{main}  to a certain statement, Proposition \ref{keyprop}, on central
extensions of the group of automorphisms of the restricted Weyl algebra. 
\subsection{$\cD_{X, [\eta],  h}$ via formal geometry.}\label{Dformal}
 Let
$A_h^{\flat}$ be the reduced Weyl algebra in $4n$ variables, that is the $k[[h]]$-algebra generated by variables $x_i, y_i, v_i, u_i$, $(1\leq i,j, \leq n)$, subject to the relations
\begin{equation}\label{weild}
\begin{split}
v_i x_j - x_j v_i =  u_i y_j - y_j u_i =\delta_{ij} h, \\
v_i y_j - y_j v_i = u_i x_j - x_j u_i = v_i u_j - u_j v_i = y_i x_j - x_j y_i =0 \\
  x_i^{p} = y_i^{p}=  v_i^{p} = u_i^{p} =0.
\end{split}
\end{equation} 
We shall identify $A_h^{\flat}$ with the central reduction $D_{\Spec A_0, \eta_A, h} (= D_{\Spec A_0, 0, h})$, $\eta_A =\eta=\sum y_idx_i $, of the algebra $D_{\Spec A_0, h}\subset D_{\Spec A_0}[[h]]$ spanned by $A_0$ and 
$h T_{\Spec A_0}$. In particular, the group scheme $\underline{\Aut} (A_0)$ acts on   $A_h^{\flat}$:
\begin{equation}\label{empty}
\psi_{can}: \underline{\Aut} (A_0)\to \underline{\Aut}_{k[[h]]}(A_h^{\flat})=: G^{\flat},\quad  g\mapsto \psi_{can, g}.
\end{equation}
We define a homomorphism
\begin{equation}\label{action1}
\psi: G_0 \to G^{\flat},
\end{equation}
to be the restriction of $\psi_{can}$ to $G_0 \subset   \underline{\Aut} (A_0) $  twisted  by a $1$-cocycle 
$$G_0 \to G^{\flat},  \quad g\mapsto \phi_{g^* \eta - \eta}.  $$
Namely, for any $k$-algebra $R$, an $R$-point  of $G_0$
is an automorphism $g$  of the  $R$-algebra $A_0\otimes R$ such that the $1$-form $\mu:= g^* \eta - \eta \in \Omega_{A_0\otimes R/R} ^1$, $\eta=\sum y_i dx_i$, is exact. 
Let $\phi_\mu: A_h^{\flat} \hat{\otimes} R \to A_h^{\flat} \hat{\otimes} R$ be the $R[[h]]$-algebra automorphism given by the formulas
$$\phi_{\mu}(x_i)=x_i,  \quad \phi_{\mu}(y_i)=y_i,  $$
 $$\phi_{\mu}(v_i)= v_i +\iota_{ \frac{\partial}{\partial x_i}} \mu,  \quad \phi_{\mu}(u_i)= u_i +\iota_{ \frac{\partial}{\partial y_i}} \mu.\footnote{Let us verify that $\phi_\mu$ is an algebra automorphism. The fact that the formulas above define an automorphism of  $\cD_{\Spec A_0, h}\hat{\otimes} R$ is clear because $\mu$ is closed. To check that this automorphism descends to $\cD_{\Spec A_0, [0], h}\hat{\otimes} R$ we need to show that the following identities hold $\cD_{\Spec A_0, [0], h}\hat{\otimes} R$.
 $$(v_i +\iota_{ \frac{\partial}{\partial x_i}}\mu )^p=  (u_i +\iota_{ \frac{\partial}{\partial y_i}}\mu)^p =0.$$ 
 Using the Katz formula  \cite[\S 7.22]{katz}  and the exactness of $\mu$ we find  that
 $$(v_i +\iota_{ \frac{\partial}{\partial x_i}}\mu )^p= v_i ^p + (\iota_{ \frac{\partial}{\partial x_i}}\mu)^p= (\iota_{ \frac{\partial}{\partial x_i}}g^*\eta)^p -  (\iota_{ \frac{\partial}{\partial x_i}}\eta)^p =0,$$
 becase $\eta$ vanishes at the origin. The second relation is proven similarly. 
           }$$
 Define (\ref{action1}) by the formula
 $$\psi_g =\phi_{g^* \eta - \eta} \circ   \psi_{can, g}.$$ 
 We claim that (\ref{action1}) is a homomorphism. Indeed, one has that
 $$ \psi_{can, g} \circ \phi_\mu \circ \psi_{can, g}^{-1}= \phi_{g^*\mu}.$$
 Using this formula we find
 $$\phi_{g_1^* \eta - \eta} \circ   \psi_{can, g_1} \circ \phi_{g_2^* \eta - \eta} \circ   \psi_{can, g_2}= \phi_{g_1^* \eta - \eta} \circ \phi_{g_1^* g_2^* \eta - g^*_1\eta} \circ  \psi_{can, g_1}  \circ   \psi_{can, g_2}$$
 and the claim follows. 
 
 The key assertion of this subsection is the following.
 \begin{lm}\label{lemmalocD_h}
Let $(X, \omega)$ be a symplectic variety with a  restricted Poisson structure $[\eta]$. Then  one has an isomorphism of $\cO_{X'}[[h]]$-algebras:
$$\cM_{X, [\eta]} \hat{\times} ^{G_0} A_h^{\flat} \iso \cD_{X, [\eta], h},$$
where the action of $G_0$ on $A_h^{\flat}$ is given  by (\ref{action1}).
\end{lm}
\begin{proof}
		Let $\pi : \cM_{X, [\eta]} \rightarrow X^{\prime}$ be the projection. For a morphism $u\colon T \to X'$, we shall denote by $u^* \cD_{X, h, [\eta]}$ the pullback $\cD_{X, h, [\eta]}$, viewed as a coherent sheaf on
		the formal scheme $X'[[h]]$, along the morphism $T[[h]] \to X'[[h]]$ induced by $u$. 
	It suffices to check that for every $S$-point $f$ of $\cM_{X, [\eta]}$ there is an isomorphism 
	\[
	\alpha_f : \mathcal{O}_S \hat{\otimes} A_h^{\flat} \cong (\pi \circ f)^* \cD_{X, h, [\eta]}
	\]
	
	such that diagram 
	\[
	\begin{tikzcd}
	\mathcal{O}_S \hat{\otimes} A_h^{\flat} \arrow[r, "\alpha_f"]\arrow[d, "\psi(g)"]& (\pi \circ f)^* \cD_{X, h, [\eta]}\arrow[equal]{d}   \\
	\mathcal{O}_S \hat{\otimes} A_h^{\flat} \arrow[r, "\alpha_{gf}"] & (\pi \circ gf)^* \cD_{X, h, [\eta]}
	\end{tikzcd}
	\]
	is commutative for every $g\in G_0(S)$.
	
	Construct $\alpha_f$ as follows. 
	By definition of $\cM_{X, [\eta]}$ the point $f$ determines an isomorphism $S \times \Spec(A_0)  \iso S \times_{X^{\prime}} X $  also denoted by $f$ fitting into
	 the commutative diagram
	
	\[
	\begin{tikzcd}
	S \times \Spec(A_0) \arrow[r, "\cong f"]\arrow[dr, ""]& S \times_{X^{\prime}} X \arrow[r, ""]\arrow[d, "pr_{S}"]& X \arrow[d, "F"]  \\
	& S \arrow[r, "\pi \circ f"] &  X^{\prime}.
	\end{tikzcd}
	\]
		
	This induces an isomorphism of the corresponding algebras of differential operators.
	\[
	\cO_S \otimes D_{\Spec(A_0)} \cong D_{S \times \Spec(A_0) / S} \cong (pr_S)_* D_{S \times_{X^{\prime}} X / S} \cong (\pi \circ f)^*  F_* D_{X / X^{\prime}} \cong (\pi \circ f)^*  F_*D_X.
	\]
	
	Applying the Artin-Rees construction we get
	
	\[
	f_*: \cO_S \otimes D_{A_0, h} \iso(\pi \circ f)^*D_{X,h}
	\]
	First, we assume that the class $[\eta]$ is represented by a global $1$-form $\eta$. Then the sheaf  $\cD_{X, h, [\eta]}$ is obtained from $D_{X, h}$ as 
	the $h$-completion of the quotient $D_{X, h}/I_{\Gamma_{\eta}}D_{X, h} $
	(see formula (\ref{graphrestriction})). The algebra $A_h^{\flat}$ is  the quotient  $ D_{A_0, h}/I_{\Gamma_{\eta_A}}D_{A_0, h}$. The desired isomorphism  $\alpha_f$ is defined from the commutative diagram.
   $$
\def\normalbaselines{\baselineskip20pt
\lineskip3pt  \lineskiplimit3pt}
\def\mapright#1{\smash{
\mathop{\to}\limits^{#1}}}
\def\mapdown#1{\Big\downarrow\rlap
{$\vcenter{\hbox{$\scriptstyle#1$}}$}}
\begin{matrix}
 \cO_S \otimes D_{A_0, h} &\rar{f_* \circ \phi_{\eta_A - f^*\eta}} &    (\pi \circ f)^*D_{X,h} \cr
\mapdown{}  &&  \mapdown{}\cr
 \mathcal{O}_S \hat{\otimes} A_h^{\flat} & \rar{\alpha_f} &    \cD_{X, h, [\eta]}.
\end{matrix}
 $$
One checks that $\alpha_f$  is independent of the choice of representative $\eta$ for $[\eta]$. Therefore, covering $X'$ by open  subsets where $[\eta]$ is represented by a $1$-form, we can patch $\alpha_f$ from local pieces.  The compatibility with the action of $G_0$ is straightforward.
\end{proof}

\subsection{Central extensions of $G$.}\label{subsectionsext}
Consider the action  of the $k[[h]]$-algebra $A_h$  on the free $k[[h]]$-module  $k[x_1,  \cdots, x_{n}]/(x_1^p, \cdots, x_{n}^p)[[h]]$ given by the formulas 
$$x_i\mapsto \text{multiplication by} \;x_i, \quad y_i \mapsto h \frac{\partial}{\partial x_i}.$$
It is well known ({\it e.g.} see \cite[Lemma 2.2.1]{bmr}) and easy to verify that this action 
defines an isomorphism  of  $k((h))$-algebras 
\begin{equation}\label{matrices}
A_h(h^{-1}) \iso \Mat_{p^n}(k((h)). 
\end{equation} 
For any $k$-algebra $R$, isomorphism  (\ref{matrices}) gives rise to a natural homomorphism
$$G(R) = \Aut_{R[[h]]} (A_h\hat{\otimes} R) \mono \Aut_{R((h))} ( \Mat_{p^n}(R((h))))\iso \PGL(p^n, R((h))).$$
This defines an embedding 
\begin{equation}\label{bkemb}
G \mono L\PGL(p^n),
\end{equation}
where $L\PGL(p^n)$ is the  loop group of $PGL(p^n)$,  that is a sheaf of groups on the category of affine schemes of over $k$ equipped with the {\it fpqc} topology sending $k$-scheme $\Spec R$ to 
$\PGL(p^n, R((h)))$. The natural morphism of algebraic groups $\GL(p^n) \to   \PGL(p^n)$ gives rise to a morphism of the loop groups $ L\GL(p^n) \to   L\PGL(p^n)$. 
By part (i) of Proposition \ref{apploopspgl} in the pullback diagram of {\it fpqc} sheaves 
   $$
\def\normalbaselines{\baselineskip20pt
\lineskip3pt  \lineskiplimit3pt}
\def\mapright#1{\smash{
\mathop{\to}\limits^{#1}}}
\def\mapdown#1{\Big\downarrow\rlap
{$\vcenter{\hbox{$\scriptstyle#1$}}$}}
\begin{matrix}
 \widetilde G: = G \times_{L\PGL(p^n) }  L\GL(p^n) & \to &   L\GL(p^n)  \cr
\mapdown{}  &&  \mapdown{}\cr
 G  & \mono &    L\PGL(p^n) 
\end{matrix}
 $$
the left vertical arrow is surjective. Thus, we have a central extension of {\it fpqc}  sheaves
 \begin{equation}\label{fundcentralextensionbis}
1 \to L\bG_m \to \widetilde G  \to   G \to 1. 
\end{equation}

Recall from \S \ref{subsectionautomorphisms} that subgroup $G^{\geq 1}\subset G$ of automorphisms identical modulo $h$ consists of inner automorphisms. Therefore, extension
 (\ref{fundcentralextensionbis}) fits into a commutative diagram 
   $$
\def\normalbaselines{\baselineskip20pt
\lineskip3pt  \lineskiplimit3pt}
\def\mapright#1{\smash{
\mathop{\to}\limits^{#1}}}
\def\mapdown#1{\Big\downarrow\rlap
{$\vcenter{\hbox{$\scriptstyle#1$}}$}}
\begin{matrix}
 1 &  \to   &  L^+\bG_m   & \to  &  \underline {A}_h ^*  & \to &  G^{\geq 1} & \to & 1\cr
&& \mapdown{}  &  &\mapdown{}  &&  \mapdown{} && \cr
1 &  \to   &  L\bG_m & \to  & \tilde G & \to &  G & \to & 1 \cr
&& \mapdown{}  &  &\mapdown{}  &&  \mapdown{} && \cr
1 &  \to   & \Gra_{\mathbb{G}_m} & \to  & \tilde G_0 & \to &  G_0 & \to & 1,  
\end{matrix}
 $$
where $ \widetilde G_0  = \widetilde G/  \underline {A}_h ^* $ and $\Gra_{\mathbb{G}_m}=  L\bG_m/ L^+\bG_m$ is the affine grassmannian for $\bG_m$.

\begin{rem}\label{fundextliealg}
Consider the central extension of the Lie algebras 
$$0 \to    \Lie \Gra_{\mathbb{G}_m} \to  \Lie  \widetilde G_0  \to \Lie  G_0 \to 0$$
corresponding to the bottom line in the diagram above. Identify the Lie algebra of the affine grassmannian with the vector space $h^{-1} k[h^{-1} ]$ of polynomial vanishing at the origin equipped with the trivial Lie bracket.   
	It is shown in \cite{bk} that the Lie algebra of $G_0$ consists of Hamiltonian vector fields on $A_0$, that is $\Lie  G_0 = A_0/k$, where the Lie bracket is induced by the Poisson bracket on $A_0$. 
	Then $\Lie  \widetilde G_0$ is isomorphic to the direct sum of  Lie algebras $A_0 \oplus h^{-2} k[h^{-1} ]$ with map to  $\Lie  G_0$  given by the projection to the first summand followed by  $A_0 \to A_0/k$.
	
\end{rem}
\begin{proof}
	It suffices to construct 
	a morphism of extensions
	  $$
\def\normalbaselines{\baselineskip20pt
\lineskip3pt  \lineskiplimit3pt}
\def\mapright#1{\smash{
\mathop{\to}\limits^{#1}}}
\def\mapdown#1{\Big\downarrow\rlap
{$\vcenter{\hbox{$\scriptstyle#1$}}$}}
\begin{matrix}
 0 &  \to   &  k  & \to  &  A_0  & \to &  A_0/k & \to & 0 \cr
&& \mapdown{h^{-1}}  &  &\mapdown{}  &&  \mapdown{\simeq} && \cr
0&  \to   & \Lie \Gra_{\mathbb{G}_m} & \to  & \Lie \widetilde G_0 & \to &  \Lie G_0 & \to & 0.  
\end{matrix}
 $$
Define a map of Lie algebras  $A_0 \to \Lie(\widetilde G_0)$ sending $a \in A_0$ to the image of  $1 + \frac{ \epsilon \tilde{a} }{h} \in   \widetilde{G}(k[\epsilon]/\epsilon^2)$
under the homomorphism $  \widetilde{G}(k[\epsilon]/\epsilon^2) \to \widetilde G_0(k[\epsilon]/\epsilon^2)$.  Here  $\tilde{a} \in A_h$ is any lifting of $a$. 
The formula 
$\Ad_{1 + \frac{\epsilon \tilde{a}}{h}} = \Id  +   \frac{\epsilon }{h} \ad_{\tilde a}$ shows that  $A_0 \to \Lie(\widetilde G_0)$  lifts $A_0 / k \to \Lie(G_0)$ as desired. 

\end{proof}

Applying the same construction to the algebra $A_h^{\flat}$ and to its representation on $k[x_1,  \cdots, x_{n}, y_1,  \cdots, y_{n}]/(x_1^p, \cdots, x_{n}^p, y_1^p, \cdots, y_{n}^p)[[h]]$ we construct a commutative diagram 
   $$
\def\normalbaselines{\baselineskip20pt
\lineskip3pt  \lineskiplimit3pt}
\def\mapright#1{\smash{
\mathop{\to}\limits^{#1}}}
\def\mapdown#1{\Big\downarrow\rlap
{$\vcenter{\hbox{$\scriptstyle#1$}}$}}
\begin{matrix}
 1 &  \to   &  L^+\bG_m   & \to  &  \underline{A_h}^{\flat *}  & \to &  G^{\geq 1 \flat} & \to & 1\cr
&& \mapdown{}  &  &\mapdown{}  &&  \mapdown{} && \cr
1 &  \to   &  L\bG_m & \to  & \widetilde G^\flat & \to &  G^\flat & \to & 1 \cr
&& \mapdown{}  &  &\mapdown{}  &&  \mapdown{} && \cr
1 &  \to   &  \Gra_{\mathbb{G}_m} & \to  & \widetilde G_0^\flat & \to &  G_0^\flat & \to & 1. 
\end{matrix}
 $$
Recall from (\ref{action1}) the homomorphism $\psi: G_0 \to G^\flat $. Denote by $\psi_0: G_0 \to G_0 ^\flat $ its composition with  the projection $G^\flat  \to G^\flat _0$.
The key step in the proof of our main theorem is the following result.
\begin{pr}\label{keyprop}
There is a unique homomorphism $\tilde \psi_0$ making the following diagram commutative.
   $$
\def\normalbaselines{\baselineskip20pt
\lineskip3pt  \lineskiplimit3pt}
\def\mapright#1{\smash{
\mathop{\to}\limits^{#1}}}
\def\mapdown#1{\Big\downarrow\rlap
{$\vcenter{\hbox{$\scriptstyle#1$}}$}}
\begin{matrix}
 1 &  \to   &   \Gra_{\mathbb{G}_m} & \to  & \widetilde{G}_0 & \to &  G_0 & \to & 1\cr
&& \mapdown{Id}  &  &\mapdown{\tilde\psi_0}  &&  \mapdown{\psi_0} && \cr
1 &  \to   &   \Gra_{\mathbb{G}_m} & \to  & \widetilde G_0^\flat & \to &  G_0^\flat & \to & 1
\end{matrix}
 $$

\end{pr}

\begin{rem}
	Let us describe  the morphism of the Lie algebras induced by  $\tilde{\psi_0}$.
 Namely, define
	$d(\tilde{\psi_0}) : f \mapsto f + \eta(H_f) + \mathcal{H}_f$. Here $H_f$ stands for the Hamiltonian vector field on $A_0$, while $\mathcal{H}_f$ is the same vector field viewed as a function on $A_0^{\flat}$. 
\end{rem}
 We end this subsection with a reformulation of Proposition \ref{keyprop} that will be used in  forthcoming paper \cite{bktv}.  
 Set 
 $B_h=  A_h \otimes_{k[[h]]} A_h^{\flat, op}$.  Let
 $$G^\sharp \subset \underline{\Aut} (B_h) \to  \underline{\Aut} (B_0)$$
 be the preimage of $\Gamma_{\psi_0}: G_0 \mono  \underline{\Aut} (B_0)= \underline{\Aut} (A_0 \otimes A_0^\flat)$, $\Gamma_{\psi_0}(g)= g\otimes \psi_0(g)$.  Proposition \ref{keyprop}  implies that the canonical extension of $ \underline{\Aut} (B_h)$ by $L\bG_m$ ({\it cf.} (\ref{fundcentralextensionbis})) restricted to $G^\sharp$ admits a unique reduction to $L^+\bG_m$:
 \begin{equation}\label{sharpextension}
 1\to L^+\bG_m \to \widehat G^\sharp  \to  G^\sharp \to 1.
 \end{equation}
 \begin{cor}\label{refforbktv}
 There exists a unique (up to a unique isomorphism) triple $(\widehat G^\sharp, \alpha, i)$ displayed in the digram
 \begin{equation}\label{diag.f.extbdouble} 
	\begin{tikzcd}
	&  & \underline{B_h(h^{-1})}^*  &  &  \\
	1 \arrow[r, ""] & \underline{B}_h^*  \arrow[r, ""]\arrow[dr, "Ad"] \arrow[ur, ""] & \widehat G^\sharp  \arrow[r, ""]\arrow[d, "\alpha"]\arrow[u, "i"]& G_0\arrow[r, ""] \arrow[d, "\Gamma_{\psi_0}"] &1  \\
	&& \underline{\Aut} (B_h)   \arrow[r, ""]  & \underline{\Aut} (B_0) &
	\end{tikzcd}
	\end{equation} 
where the north east arrow  is the natural inclusion, $i$ is a monomorphism  and $\alpha(g)= \Ad_{i(g)}.$ In addition, if $W$ is an irreducible representation of  $B_h(h^{-1})$,  $B_h(h^{-1}) \iso \End_{k((h))}(W)$, there exists  a 
   $k[[h]]$-lattice  $\Lambda \subset W$, invariant under the $B_h$-action on $W$ and under the action of $\widehat G^\sharp$:  
   $$i:  \widehat G^\sharp \mono L^+\GL(\Lambda) \subset  L\GL(W) =\underline{B_h(h^{-1})}^*.$$
\end{cor}
\begin{proof}
The diagram is merely a rearrangement of (\ref{sharpextension}). The existence of a lattice $\Lambda$ stable under  $\widehat G^\sharp$ follows Proposition \ref{subgroupsofLGL}. Since $\underline{B}_h^* \subset \widehat G^\sharp$ the lattice $\Lambda$ is $B_h$-invariant.
\end{proof}
 \subsection{Proposition \ref{keyprop} implies the main theorem.}\label{ss:reductiontopositiveloops} In this subsection we prove Theorem \ref{main} assuming Proposition \ref{keyprop}.

For the first part, let us start by reinterpreting the construction of the algebra $\cO_h \otimes_{\cO_{X'}[[h]]} \cD_{X, [\eta], h}^{op}$. Consider the homomorphism 
\begin{equation}\label{tensorprodloc}
G \to \Aut_{k[[h]]}(A_h) \times \Aut_{k[[h]]}(A_h^{\flat, op}) \mono \Aut_{k[[h]]}(A_h\otimes_{k[[h]]} A_h^{\flat, op}) 
\end{equation}
whose first component is the identity map and whose second component is the composition $G \rar{} G_0 \rar{\psi} \Aut_{k[[h]]}(A_h^{\flat})= \Aut_{k[[h]]}(A_h^{\flat, op})$\footnote{
Note that $\Aut_{k[[h]]}(A_h^{\flat, op})$ is equal, as a subgroup of the group of automorphisms of the $k[[h]]$-module $A_h^{\flat}$, to $\Aut_{k[[h]]}(A_h^{\flat})$.}.
Homomorphism (\ref{tensorprodloc}) defines a sheaf of $\cO_{X'}[[h]]$-algebras $\cM_{X, \cO_h, s} \times _G (A_h\otimes_{k[[h]]} A_h^{\flat, op})$.
By Lemma \ref{lemmalocD_h}, we have an isomorphism
$$\cM_{X, \cO_h, s} \times _G (A_h\otimes_{k[[h]]} A_h^{\flat, op}) \iso \cO_h \otimes_{\cO_{X'}[[h]]} \cD_{X, [\eta], h}^{op}.$$
Next, the $k((h))$-algebra   $(A_h\otimes_{k[[h]]} A_h^{\flat, op})(h^{-1})$ is isomorphic to the matrix algebra $\End_{k((h))}(V)$, for some vector space over $k((h))$ of dimension $p^{3n}$.
\begin{equation}\label{tensorprodlocbis}
(A_h\otimes_{k[[h]]} A_h^{\flat, op})(h^{-1})\iso  \End_{k((h))}(V)
\end{equation}
Isomorphisms (\ref{tensorprodlocbis}) and (\ref{tensorprodloc}) give rise to a homomorphism
\begin{equation}\label{homtopgl}
G \to L\PGL(p^{3n})
 \end{equation}
and, consequently, to an extension of $G$ by $L\bG_m$. Proposition \ref{keyprop} asserts that this extension admits a unique reduction to $L^+\bG_m$.
\begin{equation}\label{tmfe}
 1 \to  L^+\bG_m \to \widehat  G  \to G  \to  1.
 \end{equation}
 Thus, by part (ii) of Proposition \ref{apploopspgl}, it follows that homomorphism (\ref{homtopgl}), possibly after conjugation by an element of  $\PGL(p^{3n}, k((h)))$, factors through  $L^+\PGL(p^{3n})\subset L\PGL(p^{3n})$. 
 In the other words, there exists a $k[[h]]$-lattice $\Lambda \subset V$ such that the action of $G$ on  $(A_h\otimes_{k[[h]]} A_h^{\flat, op})(h^{-1})$ preserves $\End_{k[[h]]}(\Lambda)$:
 \begin{equation}
 A_h\otimes_{k[[h]]} A_h^{\flat, op} \subset (A_h\otimes_{k[[h]]} A_h^{\flat, op})(h^{-1}) \iso \End_{k((h))}(V) \supset \End_{k[[h]]}(\Lambda).
 \end{equation}
 The homomorphism 
 $$G \to \Aut (\End_{k[[h]]}(\Lambda)) \iso  L^+\PGL(p^{3n})$$ 
 and the $G$-torsor $\cM_{X, \cO_h, s}$ give rise to an Azumaya algebra  
 $$\cO^\sharp_h: = \cM_{X, \cO_h, s} \times _G \End_{k[[h]]}(\Lambda)$$
 which by construction coincides with $\cO_h \otimes_{\cO_{X'}[[h]]} \cD_{X, [\eta], h}^{op}$ after inverting $h$. This proves part (i) of the Theorem.
 
 To prove part (ii) of the Theorem, recall from \S \ref{subsectionautomorphisms} that $G$ is acted upon by an involution
 $\tau: G\to G$.  We claim that $\tau$ lifts to extension (\ref{tmfe}),
 $$\hat \tau: \widehat  G \to \widehat  G, \quad \hat \tau ^2 =\Id,$$
 such that the restriction of $\hat \tau$ to $L^+\bG_m$  is given by the formula
 \begin{equation}\label{sigmaact}
 \hat \tau (f(h)) = f(-h)^{-1}, \quad  f(h) \in R[[h]]^*.
 \end{equation}
 Consider the homomorphism $ L^+ \bG_m \to \bG_m$ sending $f(h)\in  L^+ \bG_m(R)= R[[h]]^*$ to $f(0)\in R^*$. This fits into the following diagram of group scheme extensions.
    $$
\def\normalbaselines{\baselineskip20pt
\lineskip3pt  \lineskiplimit3pt}
\def\mapright#1{\smash{
\mathop{\to}\limits^{#1}}}
\def\mapdown#1{\Big\downarrow\rlap
{$\vcenter{\hbox{$\scriptstyle#1$}}$}}
\begin{matrix}
 1 &  \to   &   L^+\bG_m & \to  & \widehat G & \to &  G & \to & 1\cr
&& \mapdown{}  &  &\mapdown{}  &&  \mapdown{} && \cr
1 &  \to   &   \bG_m  & \to  &\widehat{ G/G^{\geq 2}  }       & \to & G/G^{\geq 2}&\to  & 1
\end{matrix}
 $$
The action of $\hat \tau$  on    $\widehat  G$ descends to $\widehat{ G/G^{\geq 2}  } $. 
 For the  $\bZ/2\bZ$-action on  $R^*$ 
 given by formula $c\mapsto c^{-1}$, we have that 
 \begin{equation}\label{cohvanish}
  H^1(\bZ/2\bZ, R^*)\iso R^*/R^{* 2}.
 \end{equation} 
 In particular, every cohomology class gets killed after a finite \'etale extension of $R$.
 It follows that the sequence of $\bZ/2\bZ$-invariants 
 \begin{equation}\label{tmfeinv}
 1 \to  \bG_m ^{\hat \tau=1}  \to (\widehat{ G/G^{\geq 2}  })^{\hat \tau=1}   \to (G/G^{\geq 2}  )^{\tau=1}  \to  1
 \end{equation}
 is exact.  Note that $ \bG_m^{\hat \tau=1} = \mu_2 = \{1, -1\}$.  We claim that (\ref{tmfeinv}) is a split extension:
 \begin{equation}\label{tmfeinvsplit}
 (\widehat{ G/G^{\geq 2}  })^{\hat \tau=1}   \iso (G/G^{\geq 2}  )^{\tau=1} \times \mu_2.  
  \end{equation}
  Indeed, the determinant homomorphism
  \begin{equation}\label{eqdeterminant}
  \widehat G \mono L^+ GL(p^{3n})\rar{\det} L^+\bG_m
  \end{equation}
  composed with the map  $L^+G_m \to \bG_m$ factors through $\widehat{ G/G^{\geq 2}  }$ and commutes with $\tau$. Hence, it defines a homomorphism 
  \begin{equation}\label{eqdeterminantred}
(\widehat{ G/G^{\geq 2}  })^{\hat \tau=1}  \to \mu_2
\end{equation}
  whose restriction to $\mu_2$ is the identity. This gives a splitting of extension (\ref{tmfeinv}). We derive from (\ref{tmfeinvsplit}) that the extension  $\widehat{ G/G^{\geq 2}  }$ has the form
   \begin{equation}\label{finalspl}
1  \to     \bG_m   \to G_0 \ltimes \underline{A}_0^* \to   G_0 \ltimes \underline{A}_0^*/\bG_m \to 1.
\end{equation}

Now we can prove that 
 \begin{equation}\label{formularestr}
 i^*[\cO^\sharp_h]= i^*(\delta(\gamma)).
 \end{equation}
 To see this, consider the gerbe  $\cS$ of splittings of the Azymaya  algebra $\cO^\sharp_h$. 
 By definition, this is a sheaf of groupoids on $(X')_{fl}$ whose sections $\cS(Z)$ over $Z\to X'$ is the groupoid of splittings of the pullback of 
$\cO^\sharp_h$  to $Z$. This is a gerbe naturally banded by the sheaf $L^+\bG_m$ meaning that the automorphism group of any object of  $\cS(Z)$
is canonically identified with $L^+\bG_m(Z)$. By construction of $\cO^\sharp_h$ and the uniqueness statement in Proposition \ref{keyprop} this gerbe is equivalent to the gerbe of liftings of $G$-torsor  $\cM_{X, \cO_h, s}$ to a $\widehat G$-torsor.
It follows that the gerbe of splittings $\overline \cS$ of the Azymaya  algebra $i^* \cO^\sharp_h$  is  equivalent to the $\bG_m$- gerbe of liftings of $G/G^{\geq 2}  $-torsor  $\cL:= \cM_{X, \cO_h, s}\times _G  G/G^{\geq 2}   $
 to a $ \widehat{G/G^{\geq 2} }$-torsor. The set of isomorphism classes of torsors over   $G/G^{\geq 2}\iso G_0 \ltimes \underline{A}_0^*/\bG_m$ lifting a given $G_0$-torsor $\cM_{X, [\eta]}$ is in bijection $\rho$ with the
 set $H^1_{et}(X', \cO_{X'}^*/\cO_{X'}^{p *}) $ of isomorphism classes of torsors over the group scheme  $\underline{A}_0^*/\bG_m  \times _{G_0}  \cM_{X, [\eta]}$. It follows from 
 (\ref{finalspl}) that given a  $G/G^{\geq 2}  $-torsor  $\cL$ the  $\bG_m$-gerbe of liftings of $\cL$ to a torsor over $ \widehat{G/G^{\geq 2} }$ is equivalent to the gerbe of liftings of
 $\underline{A}_0^*/\bG_m  \times _{G_0}  \cM_{X, [\eta]}$-torsor $\rho(\cL)$ to a torsor over $\underline{A}_0^*  \times _{G_0}  \cM_{X, [\eta]}$. 
 This proves  (\ref{formularestr}).  
 
To prove the last assertion of Theorem \ref{main} observe that
the kernel of the restriction $i^*\colon  \Br(X'[[h]]) \to \Br(X')$ is a subgroup of the group
$H^2_{Zar}(X', \bW(\cO_{X'}))$, 
where  $\bW(\cO_{X'})$ is the additive group of the ring of big Witt vectors, that is
$$\bW(\cO_{X'})= (1 + h\cO_{X'}[[h]])^*.$$
We claim that  vanishing of $H^2(X', \cO_{X'})$ implies  vanishing of $H^2(X', \bW(\cO_{X'}))$.
Indeed,  $\bW(\cO_{X'})$ is the inverse limit of the groups of truncated Witt vectors 
$$\bW(\cO_{X'}) \cong \lim_{\leftarrow} \bW_m(\cO_{X'}).$$
 Using the exact sequence 
$$ 0 \to \bW_l(\cO_{X'}) \rar{V^m}   \bW_{l+m} (\cO_{X'}) \to \bW_m(\cO_{X'}) \to 0$$ 
it follows that, for every positive integer $m$,   the group $H^2(X, \bW_m(\cO_{X'}))$ is trivial and consequently the restriction homomorphism 
 $$H^1(X, \bW_{l+m} (\cO_{X'})) \to H^1(X, \bW_m(\cO_{X'}) )$$
 is surjective, for every $l$ and $m$. 
 Hence by Proposition 13.3.1 from  \cite[Chapter 0]{ega3}, we have
 $$H^2(X', \bW(\cO_{X'}))   \iso \lim_{\leftarrow} H^2(X', \bW_m(\cO_{X'}))  =0$$
 as desired.
\begin{rem}
Observe that under the assumptions  of Theorem \ref{main}, we have that
$$ p^{3n}([\cO^\sharp_h] -  \delta(\gamma))=0.$$
Indeed,  $\cO^\sharp_h$ is an Azumaya algebra of rank $p^{6n}$ and hence its class in the Brauer group is killed by $p^{3n}$. On the other hand,  the class $\delta(\gamma))$ is killed by $p$.
\end{rem} 
 \begin{rem}
 The proof of Theorem \ref{main} shows that 
 vanishing of $H^2(X', \cO_{X'})$ implies surjectivity of the map $\rho$  (see formula (\ref{bktheorem})), which does not follow directly from the Bezrukavnikov-Kaledin theorem. 
 Indeed, from (\ref{tmfe}) we derive an extension 
 $$ 1\to \widehat{ G^{\geq 2}  } \to \widehat{ G}/\bG_m \to  G_0 \ltimes  \underline{A}_0^*/\bG_m \to 1.$$
 Consider the $G_0 \ltimes  \underline{A}_0^*/\bG_m$-torsor $\cM$ corresponding to a restricted structure $[\eta]$ and a class $\gamma  \in  H^1_{et}(X',  \cO_{X'}^*/\cO_{X'}^{* p} )$. Using that  $H^2(X', \cO_{X'})=0$ we infer that
 $\cM$ can be lifted to a $\widehat{ G}/\bG_m$-torsor $\tilde \cM$. Pushing forward the latter under the homomorphism $\widehat{ G}/\bG_m \to G$ we get a quantization with  $\rho$-invariant $\gamma$.
\end{rem}
 \begin{rem}\label{errorinbk}
 In  \cite[Proposition 1.24]{bk} the authors erroneously assert the subgroup $G\mono L\PGL(p^n)$  from (\ref{bkemb}) preserves a lattice, that is,  possibly after conjugation by an element of  $\PGL(p^{n}, k((h)))$, factors through  $L^+\PGL(p^{n})\subset L\PGL(p^{n})$.  This claim led the authors to a mistake in the statement of Proposition 1.24.
 In fact,  even the subgroup of translations $\Spec A_0=  \alpha_p^{2n} \subset G $ does not admit an invariant lattice. This follows from the fact the commutator map
 $$\Lie \alpha_p^{2n}  \otimes \Lie \alpha_p^{2n} \to \Lie L\bG_m= k((h))$$
 arising from extension (\ref{fundcentralextensionbis}) is given by the formula $\frac{1}{h}\sum _i dy_i \wedge dx_i$, {\it i.e.} does not factor through  $\Lie L^+\bG_m = k[[h]]$.

\end{rem}
 
\section{Central extensions of the group of Poisson automorphisms}\label{Skeyproposition}
In this section we prove Basic Lemma \ref{lifting} and derive from it Proposition \ref{keyprop}. For the duration of this section we fix a symplectic vector space $(V, \omega_V)$ of dimension $2n$ and denote by $A_h$ the corresponding restricted Weyl algebra, $G$ its group of automorphisms and $G_0$ the quotient of $G$ by the subgroup of automorphisms identical modulo $h$ viewed as a group scheme of automorphisms of $A_0$ preserving the class $[\eta_V]$.

\subsection{Properties of $G_0$.}\label{propg0}
Recall from \cite[Proposition 3.4]{bk} that the reduced subgroup $G_0^0=(G_0)_{red} \subset G_0$ is equal to the stabilizer of the point $\Spec k \mono \Spec A_0$: for every $k$-algebra $R$,
  $G_0^0(R)$ is the subgroup of $G_0(R)$ that consists of $R$-linear automorphisms of $A_0 \otimes R$ that preserve the kernel of the homomorphism $A_0 \otimes R \to R$ induced by $A_0 \to k$.
 According to \cite[Lemma 3.3]{bk} the Lie algebra of $G_0$ (resp. $G_0^0$)  is the algebra of all Hamiltonian vector fields\footnote{Recall that a vector field is said to be Hamiltonian if it has the form 
 $H_f$, for some $f\in A_0$.} on $\Spec A_0$ (resp. the algebra of all Hamiltonian vector fields  vanishing at $\Spec k \mono \Spec A_0$). In particular, we have
 \begin{equation}\label{dimG_0}
 \dim G_0 = \dim G^0_0 = \dim _k \Lie G^0_0 = \dim _k m^2 =  p^{2n} - 2n -1,
\end{equation}
 where $m$ is the maximal ideal in $A_0$.  
 
 Denote by $\alpha_p$ the Frobenius kernel on $\bG_a$. The finite group scheme $\alpha:= \alpha_p^{2n}$ acts on $\Spec A_0 = \alpha $ by translations  inducing the inclusion $\alpha \mono G_0$.
 Observe that the product morphism
  $$ \alpha \times G_0^0 \to G_0$$ 
  induces an isomorphism of the underlying schemes. 

\begin{lm}\label{Gisconnected} 
The group schemes $G_0$ and $G_0^0$  are connected. Moreover, there is a surjective homomorphism
$$ G_0^0 \epi \Sp(2n)$$
whose kernel is a unipotent algebraic group.
\end{lm}{}

\begin{proof}

It suffices to prove the assertions for the reduced group $G_0^0$. 
To show that $G_0^0$ is connected we consider the filtration 
$$ \ldots \subset F^2G_0^0 \subset F^1G_0^0 \subset  G_0^0$$
by normal group subschemes of $G_0^0$ and prove that all the associated quotients are connected. Namely,  for any $k$-algebra $R$, we set
  $$F^iG_0^0(R) = \{ \phi \in G_0^0(R) \, | \, \phi = \Id \mod m^{i+1} \otimes R\}.$$ 
  It is easy to see that this functor is representable by a normal group 
subscheme of $G_0^0$. 

The action of  $G_0^0$ on the tangent space $(m/m^2)^*$ preserves $\omega$. Thus it gives rise to a monomorphism
$$G_0^0/F^1G_0^0 \to \Sp(2n)$$
which is, in fact, an isomorphism because it has a section. In particular, we have that
\begin{equation}\label{dimsp}
\dim G_0^0/F^1G_0^0 = \dim \Sp(2n)= \dim_k m^2/m^3  
\end{equation}
To check that the other quotients are connected we construct injective homomorphisms 
\begin{equation}\label{pfconn}
\alpha_i : F^iG_0^0 / F^{i+1}G_0^0 \hookrightarrow \underline{m^{i+2}/m^{i+3}}, \quad i\geq 1,
\end{equation}
where $ \underline{m^{i+2}/m^{i+3}}$  is the vector group associated to the space $m^{i+2}/m^{i+3}$ and then using (\ref{dimG_0})  conclude that $\alpha_i$ are isomorphisms. For the sake of brevity we  only define
$\alpha_i$ on $k$-points.   
Take  $\phi \in F^iG_0^0(k)$ and consider $\phi^* : A_0 \rightarrow A_0$. By definition $\phi^* = \Id \mod m^{i+1}$, so
$\phi^* - \Id$ maps $m^r$ to $m^{i+r}$, for every $r\geq 0$. Hence  $\phi^* - \Id$ defines  a homogeneous degree $i$ map $\theta_\phi : \oplus m^r/m^{r+1} = A \rightarrow A$ which is, in fact, a derivation. 
Let us show that $\theta_\phi$ lies in a Lie algebra of $G_0^0$, i.e. $L_{\theta_\phi}\eta$ is exact. Indeed, since $\phi \in G_0^0$ we have $\phi^*\eta = \eta + dK$ for some $K\in A_0$. But then $L_{\theta_\phi}\eta = dK_{i+2}$, where $K_{i+2}$ is the homogeneous component of $K$ of degree $i+2$. It follows that $\theta_\phi$ is Hamiltonian:
$\iota _{\theta_\phi}  \omega = \iota _{\theta_\phi}   d \eta = d(K_{i+2} - \iota _{\theta_\phi} \eta)$. 
Set  
$$\alpha_i(\phi)= K_{i+2} - \iota _{\theta_\phi} \eta  \in m^{i+2}/m^{i+3}.$$
Using the identity $\theta_{\phi \circ \psi } = \theta_\phi +\theta_\psi $,  for  every $\phi, \psi  \in F^iG_0^0(k)$, one checks that $\alpha $ is a group homomorphism and that it factors through   $F^iG_0^0 / F^{i+1}G_0^0$.
For the injectivity of (\ref{pfconn}) observe that  $\theta_\phi =0$ if and only if $\phi \in F^{i+1}G_0^0(k)$. 
 
From (\ref{pfconn}) we have that, for every $i\geq 1$,
 $$\dim F^iG_0^0 / F^{i+1}G_0^0 \leq \dim _k m^{i+2}/m^{i+3}.$$
If for some $i$ the inequality is strict then using (\ref{dimsp}) we would have that  $\dim G_0^0  < \dim _k m^{2}$ contradicting to (\ref{dimG_0}). It follows that all $\alpha_i$ are isomorphisms as desired.
\end{proof}{}
Recall from Lemma \ref{decomp}\footnote{We remark that all the results of the Appendix, in particular, Lemma \ref{decomp}, do not depend on anything  from the main body of the paper.} a decomposition 
$$ \Gra_{\bG_m} \iso  \underline{\bZ} \times \hat \bW. $$
\begin{cor}\label{redtoW}
The extension (\ref{ext}) admits a unique reduction to $\hat \bW \subset \Gra_{\bG_m}$. Notation:
$$1\to \hat \bW \to \nG_0 \to G_0\to 1.$$
 \end{cor}
\begin{proof}
By Lemma \ref{Gisconnected} we have that $Hom(G_0, \underline{\bZ})=0$. The uniqueness part follows. For the existence,  note that the composition
 $$\widetilde G\mono L\GL(p^n) \rar{\det} L\bG_m \epi  \underline{\bZ} $$
 factors through $\widetilde G_0$. We claim that  setting $\nG_0:= \ker (\tilde G_0 \to \underline{\bZ})  $ does the job. Indeed, the only assertion that requires a proof is the surjectivity of the projection
 $\nG _0 \to G_0$. By construction, $\nG _0$ projects onto the kernel of the homomorphism $G_0 \to  \underline{\bZ}/p^n$ induced by $\widetilde G_0 \to \underline{\bZ}$. But by Lemma \ref{Gisconnected}
 every such homomorphism is trivial.
  \end{proof}

Consider $m^2\subset A_0$ as a Lie subalgebra of $A_0$ equipped with Poisson bracket. Recall that 
\begin{equation}\label{liealgebraofthereducedsubp}
m^2 \rar{f \mapsto H_f} \Lie  G_0^0
\end{equation}
is an isomorphism of Lie algebras. The grading on $A_0$ induces a grading  on the Lie algebra $m^2$:
 $$m^2 \cong \bigoplus _{2 \leq i \leq 2n(p-1)} m^i/m^{i+1}$$  
such that the Lie bracket has degree $-2$.

\begin{lm}\label{commutator}
Isomorphism (\ref{liealgebraofthereducedsubp}) induces 
$$[\Lie G_0^0, \Lie G_0^0]  \cong \bigoplus _{2 \leq i < 2n(p-1)} m^i/m^{i+1}. $$ 
In particular, $[\Lie G_0^0, \Lie G_0^0]$ has codimension $1$ in $\Lie G_0^0$.
Moreover, $\lsp(2n) = m^2/m^3$ together with any non-zero element $z\in m^3/m^4$ generate the Lie algebra  $[\Lie G_0^0, \Lie G_0^0]$.
\end{lm}

\begin{proof}
By a direct computation the Poisson  bracket of any two monomials of total degree  $2n(p-1) +2$ is $0$ {\it i.e.},  $[\Lie G_0^0, \Lie G_0^0]$  does not contain non-zero homogeneous elements 
of degree  $2n(p-1)$.
 Hence $[\Lie G_0^0, \Lie G_0^0] \subset  \bigoplus _{2 \leq i < 2n(p-1)} m^i/m^{i+1}$.
Since  $[\lsp(2n), \lsp(2n)] = \lsp(2n)$, we have that   $m^2/m^3 \subset  [\Lie G_0^0, \Lie G_0^0]$. Also it is clear that $[\Lie G_0^0, \Lie G_0^0]$ 
contains at least one non-zero element of degree $3$ ({\it e.g.}, $\{x_1^2, x_1 y_1^2\}= 2 x_1^2  y_1$). 
 To complete 
the proof of the Lemma it suffices to verify that the Lie subalgebra $\fg$   generated by  $m^2/m^3$ and a non-zero element of degree $3$ coincides with $ \bigoplus _{2 \leq i < 2n(p-1)}  m^i/m^{i+1}$. 
We check by induction on $d$ that $m^d /m^{d+1}\subset \fg$ provided that $2 \leq d < 2n(p-1)$.
The base of induction, $d = 3$, can be easily checked directly  follows from Lemma \ref{irrrepsp}.

Choose a symplectic basis $(x_i, y_j)$ for $V^*$ and let $E = x_1^{a_1}y_1^{b_1}\ldots x_n^{a_n}y_n^{b_n} \in m^d / m^{d+1}$ with $d > 3$. 

Note that 
\begin{itemize}
	
	\item $3a_ix_1^{a_i}y_i^{b_i} =\{x_i^{a_i+1}y_1^{b_i-2}, y_i^3\}$ 
	(and $-3b_ix_1^{a_i}y_1^{b_i} =\{x_i^{a_i-2}y_1^{b_i+1}, y_1^3\}$)
	
	\item $(a-1-2b)x_i^{a}y_i^b = \{ x_i^{a-1}y_i^b, x_i^2y_i\}$ (and $(2a - b + 1)x_i^{a}y_i^b = \{ x_i^{a}y_i^{b-1}, x_iy_i^2\}$)
	
	\item $- x_ix_j = \{x_iy_i, x_ix_j\}$
	
	\item $x_i^{p-1}y_i^{p-1} + 2(p-1)x_i^{p-2}y_i^{p-2}x_jy_j= \{ x_i^{p-1}y_i^{p-3}x_j, y_i^2y_j\}$
	
	\item $2(p-1)x_i^{p-2}y_i^{p-2}x_j^2y_j + 2x_i^{p-1}y_i^{p-1}x_j=\{x_i^{p-1}y_i^{p-3}x_j^2, y_i^2y_j\}$
\end{itemize}

Assume first that $p > 3$. Then if for some $i$ we have $a_i < p-1$ and $b_i \geq 2$ (or $a_i \geq 2$ and $b_i < p-1$) then by the first formula above $E$ is generated by elements of degree $2 \leq d^{\prime} < d$, which are in $\fg$ by the induction assumption. Otherwise for all $i$ the pair $(a_i, b_i)$ equals $(p-1, p-1)$, $(1, 1)$, $(1, 0)$ or $(0,1)$. 

If for some $i$ the pair is $(1, 1)$ we get that from the second formula that $E \in \fg$. If there are at least two pairs of the type $(1, 0)$ or $(0,1)$ we are done by the third formula.

Otherwise we may assume that $E = x_1^{p-1}y_1^{p-1} \ldots x_{n-r}^{p-1}y_{n-r}^{p-1}$ or $E = x_1^{p-1}y_1^{p-1} \ldots x_{n-r}^{p-1}y_{n-r}^{p-1}x_{n - r + 1}$ for some $0 <r < n$. In these cases we are done by the last two formulas.

Now assume $p = 3$. 
For $d>3$ note that if for some $i$ the number $a_i + b_i - 1$ is not divisible by 3 we are done using the second formula. So assume that for all $i$ the pair $(a_i, b_i)$ equals either $(p-1, p-1)$, $(1, 1)$, $(1, 0)$ or $(0, 1)$. In these cases we proceed as above.
\end{proof}

\begin{lm}\label{G_0}
We have a commutative diagram
	\[
\begin{tikzcd}
G_0^0 / [G_0^0, G_0^0]  \arrow[r, "\cong"] \arrow[d, "\cong"]&  \mathbb{G}_a &   \\
G_0 / [G_0, G_0]  \arrow[ur, ""]&  &
\end{tikzcd}
\]
Moreover, the projection  $G_0 \to G_0 / [G_0, G_0] $ admits a section yielding to a decomposition $G_0 \cong [G_0, G_0] \rtimes \mathbb{G}_a$. Lastly, we have that
\begin{equation}\label{lieofcomm}
\Lie [G_0^0, G_0^0]= [\Lie G_0^0, \Lie G_0^0].
\end{equation}

\end{lm}{}

\begin{proof}
Let us construct a group scheme homomorphism $\phi$ from $G_0$ to $\mathbb{G}_a$. 
For a $k$-algebra $R$ and  $g \in G_0(R)$, we have that  $g(\eta) = \eta + df \in \Omega^1_{A_0\otimes R/R}$, for some $f \in \coker (R \to A_0 \otimes R)$. Consider the element
  $$\phi(g)= [f \cdot \omega ^n] \in H_{DR}^{2n}(A_0 \otimes R/R)\iso R,$$
  where the isomorphism above is induced by $k \iso H_{DR}^{2n}(A_0)$ that takes $1\in k$ to the inverse Cartier operator applied to $\omega^n$.

To show that $\phi$ is a homomorphism consider two elements $g_1, g_2 \in G_0(R)$. Write $g_1(\eta) = \eta + df_1, g_2(\eta) = \eta + df_2$. Since $g_2 \circ g_1(\eta) = \eta + df_2 + d(g_2(f_1))$ the image of $g_2 \circ g_1$ equals $[(f_2 + g_2(f_1)) \cdot (\omega)^n]$. On the other hand, $\phi(g_1) + \phi(g_2) = [(f_2 + f_1) \cdot (\omega)^n]$.
Thus it suffices to prove that $G_0$ acts trivially on $H_{DR}^{2n}(A_0)$. We claim that, in fact, every $1$-dimensional representation of $G_0$ is trivial. Indeed,  $G_0$ is generated by two subgroups 
 $\alpha =\alpha_p^{2n}$  and  $G_0^0$. Since $\alpha $ has no nontrivial homomorphisms to $\mathbb{G}_m$ it suffices to prove the assertion for $G_0^0$. By Lemma \ref{Gisconnected} $G_0^0$ is an extension of 
 $\Sp(2n)$ by a unipotent group and neither of the two groups has nontrivial $1$-dimensional representations. This proves that $\phi$ is homomorphism. 
 
 The restriction of $\phi$ to  $G_0^0$ yields a homomorphism
 \begin{equation}\label{comcomp}
 G_0^0 / [G_0^0, G_0^0] \to \mathbb{G}_a.
 \end{equation}
Next, we shall construct a homomorphism $s: \mathbb{G}_a   = \Spec k[t] \to  G_0^0$ whose composition with the projection $G_0^0  \to G_0^0 / [G_0^0, G_0^0]$ followed by (\ref{comcomp}) is $\Id$.
Set $u = \prod x_i^{p-1}   \prod y_i^{p-1} \in A_0$. Define $s(t) \in Aut_{k[t]} (A_0[t]) $ sending $f \in A_0$ to $f- \frac{t}{2} \{f,u\}$. One verifies directly that $s$ is group homomorphism and a section of   (\ref{comcomp}).
Let us check that  (\ref{comcomp}) is an isomorphism. First, from lemma \ref{commutator} we know that $[\Lie G_0^0, \Lie G_0^0]$ has codimension  $1$ in  $\Lie G_0^0$. Secondly, since $G^0_0=(G_0)_{red}$ is smooth, both groups $[G_0^0, G_0^0]$ and $G_0^0 / [G_0^0, G_0^0]$ are also smooth.
Moreover, we have that  $[\Lie G_0^0, \Lie G_0^0] \subset  \Lie [G_0^0, G_0^0]$ (see {\it e.g.} \cite[Prop. 3.17]{b} ). It follows  that the dimension of $G_0^0 / [G_0^0, G_0^0]$ is at most $1$. Thus   (\ref{comcomp})
is a homomorphism from a smooth connected algebraic group of dimension $\leq 1$ to $\mathbb{G}_a$ and as we have already seen this homomorphism admits a section. It follows that (\ref{comcomp}) is an isomorphism.
This also proves formula (\ref{lieofcomm}).

To complete the proof of Lemma it suffices to check that  the homomorphism $G_0^0 / [G_0^0, G_0^0] \to  G_0 / [G_0, G_0]  $ is surjective. Since $G_0$ is generated by $\alpha$ and $G^0_0$ it is enough to show that
$\alpha \in   [G_0, G_0] $. Consider the subgroup  $P= \alpha \rtimes \Sp(2n) \subset G_0$. We claim that  $[P,P]=P$. Indeed, there is a surjection $[P, P] \epi [\Sp(2n), \Sp(2n) ] = \Sp(2n) $. The kernel of the surjection is a subgroup of $\alpha$ whose Lie algebra is a $\Sp(2n)$-invariant subspace of $\Lie \alpha $. It follows that the kernel is either trivial which clearly not the case or equal to $\alpha$ as desired.   
 
\end{proof}{}

Next, we shall show that $[G_0, G_0] $ is generated by $\Sp(2n)$, $\alpha$, and a certain one-parameter subgroup $\bG_a \subset    G_0^0$. We start with the following observation.
\begin{lm}\label{1-subgroups} 
Let $f\in A_h$ be an element of the restricted Weyl algebra such that $f^{\frac{p+1}{2}}=0$.  Consider 
the homomorphism 
$$\tilde \lambda_f: \bG_a = \Spec k[\tau]   \to \underline{A_h(h^{-1})}^*$$  
given by the formula
$$e^{\frac{\tau f}{h}}= \sum_{i=0}^{ \frac{p-1}{2}} \frac{(\tau f)^i}{h^i i!}.$$
Then the $p$-th power of the operator $\ad_{\frac{\tau f}{h}} : A_h \hat{\otimes} k[\tau] \to A_h \hat{\otimes} k[\tau]$ is zero and
\begin{equation}\label{advAd}
\Ad_{e^{  \frac{\tau f}{h} } } = \sum_{i=0}^{ p-1} \frac{\ad_{\tau f}^i}{h^i i!} =:e^{\ad_{\frac{\tau f}{h}}}.
\end{equation}
 In particular, $e^{\frac{\tau f}{h}}$ normalizes the lattice  $\underline{A_h} \subset \underline{A_h(h^{-1})}$ and, thus, defines a homomorphism
 \begin{equation}\label{emG_aintildeG}
 \tilde \lambda_f: \bG_a   \to \tilde G.  
 \end{equation}
	\end{lm}
\begin{proof}
The only assertion that requires a proof is formula (\ref{advAd}). Both sides of the equation can be thought as homomorphisms from $\bG_a$ to the  loop group of $k((h))$-linear automorphisms of
$ A_h(h^{-1})$. One readily sees that the differentials of these homomorphisms at $\tau =0$ are equal. It follows that the homomorphisms are equal on the subscheme $\alpha_p \subset \bG_a$. Also, by the assumption on $f$,
both homomorphisms  are given by matrices in $\End_{k[\tau]((h))}( A_h \hat{\otimes} k[\tau] (h^{-1}))$ whose entries are polynomials in $\tau$ of degree less than $p$. Therefore, the homomorphisms are equal on $\bG_a$.
\end{proof}
Let $(x_i, y_j)$ be a symplectic basis for $V^*$. 
For $p>3$, define a homomorphism 
\begin{equation}\label{emG_a}
\lambda: \bG_a =\Spec k[\tau]  \mono G^0_0
\end{equation}
by the following equations
$$\lambda(\tau, x_i) = x_i, \quad \text{for\; all}\; i$$
$$\lambda (\tau, y_1) = y_1 +3 \tau x_1^2, \quad  \lambda(\tau, y_i) = y_i, \; \text{for\; all}\; i\ne 1.$$
The differential of $ \lambda $ is the Hamiltonian vector field $H_{-x_1^3}$. The construction from Lemma \ref{1-subgroups}  gives a lifting $ \tilde \lambda_{x_1^3}= e^{\frac{\tau x_1^3}{h}}: \bG_a   \to \widetilde G$ of $\lambda$.  

For $p = 3$ define $\lambda$ by

$$\lambda(\tau, x_i) = x_i, \quad \lambda(\tau, y_i) = y_i \quad \text{for\; all}\; i\ne 1$$
$$\lambda (\tau, x_1) = x_1 + \tau x_1^2, \quad  \lambda (\tau, y_1) = y_1 -2 \tau x_1 y_1 + 2\tau^2x_1^2y_1.$$
The differential of $\lambda$ in this case is $H_{-x_1^2y_1}$. The homomorphism $ \tilde \lambda_{x_1^3}= e^{\frac{\tau x_1^2y_1}{h}}: \bG_a   \to \tilde G$ lifts $\lambda$.
 
\begin{lm}\label{H} The group scheme  $H: = [G_0, G_0]$ is generated by $\Sp(2n)$, $\alpha$, and the image of $\lambda$.
	\end{lm}
\begin{proof}
First, we show that  $\alpha $ and $[G_0^0, G_0^0]$ generate $H$.
Indeed, since $\alpha \subset H $, we have that, for any $k$-algebra $R$ 
$$H(R)=\alpha(R)(G_0^0(R) \cap H(R)).$$
Thus	it suffices to prove that $G_0^0(R) \cap H(R) = [G_0^0, G_0^0](R)$. 
	By Lemma  \ref{G_0} $G_0^0 / [G_0^0, G_0^0] \cong G_0 / H$, so the assertion holds.
	
Thus it remains to prove that 	$[G_0^0, G_0^0]$ is generated by $\Sp(2n)$ and the image of $\lambda$. Since the groups in question are smooth it suffices to verify that $\Lie [G_0^0, G_0^0]$ 
is generated by $\lsp(2n)$ and  $\Lie \lambda(\bG_a)$. But this is immediate from Lemmas   \ref{commutator} and  \ref{G_0}.
\end{proof}

Consider the extension
\begin{equation}\label{ext2}
1\to \hat \bW \to \nG_0 \to G_0\to 1
\end{equation}{}
from Lemma \ref{redtoW}.
\begin{lm}\label{sectionlemma}
The restriction of the extension~\eqref{ext2} to $G_0^0$ admits a unique splitting, that is there exists a unique homomorphism $G_0^0 \to \nG_0$ whose composition with the projection  $\nG_0 \to G_0$ is the identity.
\end{lm}{}

\begin{proof}
 Recall from (\ref{fundcentralextensionbis}) the extension 
$$1 \to L\bG_m \to \widetilde G  \to   G \to 1. $$
The kernel  $\underline {A}_h ^*/\bG_m$ of surjection $\pi: G \to G_0$  is a pro-unipotent group scheme. Thus, by part (iii) of Proposition \ref{apploopspgl} the restriction of the above extension  to $\pi^{-1}G_0^0\mono G$ admits 
a unique reduction to $L^+  \bG_m$. Equivalently, the extension  
$$1\to \Gra_{\bG_m}  \to \widetilde G_0 \to G_0\to 1$$
admits a unique splitting $\upsilon: G^0_0 \to \tilde G_0$  over $G^0_0\subset G_0$.
It remains to show that $\upsilon$ lands in  $\nG_0$. From the proof of Lemma \ref{redtoW} $\nG_0$ is the kernel of a homomorphism $\widetilde G_0 \to \underline \bZ$. Since $G^0_0$ is connected its composition with $\upsilon$ is identically $0$ as desired.  
\end{proof}{}

Recall from Remark \ref{fundextliealg} an isomorphism of Lie algebras $\Lie  \nG_0 = \Lie \widetilde G_0 \iso   A_0 \oplus h^{-2} k[h^{-1} ]$. 
Also recall an identification $\Lie  G_0^0\iso m^2\subset A_0$.
\begin{lm}\label{splittings}
	The morphism $\Lie G_0^0 \to \Lie \nG_0$ induced by the splitting from Lemma  \ref{sectionlemma} equals the composition $m^2 \mono A_0 \mono \Lie  \nG_0$.
	\end{lm}
\begin{proof}
	The difference of the two morphisms of Lie algebras is a homomorphism  from $\Lie G_0^0$ to the abelian Lie algebra $\Lie \hat \bW$.  
	 Thus it suffices to check that the morphisms coincide on the one-dimensional Lie algebra
	$\Lie G_0^0 / [\Lie G_0^0, \Lie G_0^0]$, which is immediate from Lemma \ref{G_0}.
\end{proof}

\subsection{Extensions of $\alpha $ by $\hat \bW$.}\label{alphaext}
In this subsection we shall apply the theory of restricted Lie algebras  to study the category of central extensions of the group scheme $\alpha$ by the group ind-scheme   $\hat \bW$.
Recall (see {\it e.g.}, \cite[Chapter II, \S 7]{gd}) that the Lie algebra of a group scheme $H$ over a field of characteristic $p>0$ is equipped with the $p$-th power operation giving
$\Lie H$  a restricted Lie algebra structure. We are not aware of a written account of such theory for group ind-schemes.  Therefore we shall use the following trick to reduce our problem to the well 
documented setup. 

For an affine scheme $S$ and an affine group scheme $H$ over $k$ denote by $\underline {\Mor}(S, H)$  the  {\it fpqc} sheaf of groups assigning to a scheme $T$ over $k$ the group 
$\Mor(S\times T, H)$. 

\begin{lm}
	Let $G$ be a finite connected group scheme  over $k$, $H$ a smooth commutative group scheme, and let $S$ be an affine scheme.  Then the groupoid  of central extensions of $G$ by  $\underline{\Mor}(S, H)$ in the category of  {\it fpqc} sheaves of groups 
	is equivalent to the category of central extensions of $G_S = G \times S$ by  $H_S= H\times S$ in the category of group schemes over $S$.
\end{lm}

\begin{proof}
	Assume we are given a central extension
	\begin{equation}
	1 \rightarrow H_S \rightarrow K \rar{\pi } G_S \rightarrow 1,
	\end{equation}
	and let us construct a central extension $F$ of $G$ by $\underline{\Mor}(S, H)$. For any scheme $T$ over $\Spec(k)$ define $F(T) = \{g \in K(T \times S) |  T \times S \rar{\pi(g)} G \times S \rightarrow G \text{ factors through the projection to } T\}$. It is easy to see that the resulting $F$ is a sheaf and that $\underline{\Mor}(S, H)$ injects into it, so it is left to prove that $F \rightarrow G$ is a surjection. Indeed, since the morphism $\pi$ is flat and $H_S$ is smooth we get that $K \rightarrow G_S$ is formally smooth (\cite{stacks} Lemma 29.33.3). Then since the map $\Spec(k) \times S \rightarrow G \times S$ is a nilpotent thickening we get that $\pi$ has a section. Thus $F \rightarrow G$ is surjective.

	Conversely, if we have a central extension
	\begin{equation}
	1 \rightarrow \underline{\Mor}(S, H) \rar{i} F \rightarrow G \rightarrow 1,
	\end{equation}
	define for any $S$-scheme $T$ the group $K(T)$ to be $F(T) \times^{\Mor(T \times S, H)} \Mor(T, H)$. Here the map $\Mor(T \times S, H) \rightarrow F(T)$ is induced by $i$ and $Mor(T \times S, H) \rightarrow \Mor(T, H)$ is defined to be the restriction to the graph of the structure morphism $T \to S$.  Let $H_S(T) \rightarrow F(T) \times \Mor(T, H)$ be the homomoprhism whose composition with the first projection takes $H_S(T)$ to the neutral element and whose composition with the second projection is the identity map. This defines an injection of sheaves $H_S \mono K$, making $K$ into an $H_S$-torsor over $G_S$ representable by a scheme (\cite[Chapter {\rm III}, Theorem 4.3]{Milne}). That is enough.
\end{proof}

\begin{cor}\label{equiv}
	The groupoid of central extensions of $\alpha
	$ by $\underline{\Mor}(\mathbb{A}^1, \mathbb{G}_m)$ in the category of {\it fpqc} sheaves is equivalent to the groupoid of central extensions of $\alpha   \times \mathbb{A}^1$ by $\mathbb{G}_m \times \mathbb{A}^1$ in the category of group schemes over $\mathbb{A}^1$.
\end{cor}
\begin{rem}\label{mor}
	The above groupoids are discrete {\it i.e.}, objects do not have non-trivial automorphisms. Indeed, the Cartier dual group to $\alpha  $ is isomorphic to itself. In particular, it has no non-trivial $\mathbb{A}^1$ points. Hence every homomorphism $\alpha   \times \mathbb{A}^1 \to \mathbb{G}_m \times \mathbb{A}^1$ in the category of group schemes over $\mathbb{A}^1$ is trivial.
\end{rem}

Observe that the evaluation at $0$ defines a split surjection $\underline{\Mor}(\mathbb{A}^1, \mathbb{G}_m) \to \mathbb{G}_m$, whose kernel is identified with $\hat \bW$. Hence we have a decomposition $\underline{\Mor}(\mathbb{A}^1, \mathbb{G}_m) = \hat \bW \times \mathbb{G}_m$.

Declare the restricted Lie algebra of $\Gra_G$ to be $\Lie(\Gra_{G}^0) = h^{-1}k[h^{-1}]$ with the trivial Lie bracket and the restricted power operation given by the absolute Frobenius.

\begin{Th}
	The groupoid of central extensions of the group scheme $\alpha  $ by $\hat \bW$ is equivalent to the groupoid of central extensions of the restricted Lie algebra $\Lie(\alpha  )$ by $\Lie(\hat \bW)$.  
\end{Th}

\begin{proof}
	Using that the multiplication by $p$ is $0$ on $\alpha$ and surjective on $\hat \bW$
	it follows that every extension of $\alpha$ by  $\hat \bW$ admits a reduction to 
	$\Gra^0_{\mu_p}$. Moreover, since $\Hom (\alpha, \hat \bW) =0$ (by Corollary \ref{hom}), such a reduction is unique. Thus the groupoid
	of central extensions of $\alpha$ by $\hat \bW$ is equivalent to 
	the groupoid
	of central extensions of $\alpha$ by $ \Gra^0_{\mu_p}$.
	The groupoid of central extensions of $\alpha$ by $\Gra_{\mu_p}^0$ is equivalent to a full subcategory of the groupoid of central extensions of $\alpha \times \mathbb{A}^1$ by $\mu_p \times \mathbb{A}^1$. This subcategory classifies families of central extensions whose fiber over $0 \in \mathbb{A}^1$ is a trivial extension.
	
	Next we claim that for an extension 
	\begin{equation}
	\mu_p \times \mathbb{A}^1 \rightarrow K \rightarrow \alpha   \times \mathbb{A}^1
	\end{equation}
	the $\mathbb{A}^1$-group scheme is of height $1$. Indeed, the Frobenius map $K \to K$ factors as $K \to \alpha \times \mathbb{A}^1 \to \mu_p \times \mathbb{A}^1 \to K$. From Remark \ref{mor} we conculde that this map is trivial.
	
	Thus, by \cite[Chapter {\rm II}, \S 7, Theorem 3.5]{gd} the category of central extensions of $\alpha \times \mathbb{A}^1$ by $\mu_p \times \mathbb{A}^1$ is equivalent to the category of extensions of corresponding restricted Lie algebras over $\mathbb{A}^1$. The latter is equivalent to the category of extensions of $\Lie(\alpha)$ by $\Lie(\mu_p \times \mathbb{A}^1)$ in the category of restricted Lie algebras over $k$. The above equivalence unduces an equivalence between the subcategories of central extensions of $\alpha \times \mathbb{A}^1$ by $\mu_p \times \mathbb{A}^1$ trivial over $0 \in \mathbb{A}^1$ and the category of restricted Lie algebra extensions of $\Lie(\alpha  )$ by $\Lie(\hat \bW)$.
 
\end{proof}

\begin{cor}\label{extensions}
	The groupoid of central extensions of the group scheme $\alpha  $ by $\hat \bW$ which split over any factor $\alpha_{p} \subset \alpha  $ is equivalent to the set of $\Lie(\hat \bW)$-valued skew-symmetric bilinear forms on $\Lie(\alpha)$ viewed as a groupoid with no non-trivial morphisms.
\end{cor}

\begin{proof}
	Given a central extension of $\alpha  $ by $\hat \bW$ let 
	\begin{equation}\label{extlie}
	0\to \Lie(\hat \bW) \to \cL \to \Lie(\alpha)\to 0
	\end{equation}
	be the corresponding extension of Lie algebras. The commutator on $\cL$ defines a $\Lie(\hat \bW)$-valued skew-symmetric bilinear form on $\Lie(\alpha)$. To construct the functor in the other direction set $\cL = \Lie(\hat \bW) \oplus \Lie(\alpha)$ as a vector space. The skew-symmetric form defines a Lie bracket on $\cL$ making $\cL$ a central extension of Lie algebras. Define the restricted power operation on $\cL$ by the formula $(f, g)^{[p]} = f^{[p]}$. Let us check that $\cL$ is a restricted Lie algebra. We have to check that the restricted power operation satisfies $$(X+Y)^{[p]}=X^{[p]}+Y^{[p]}+\sum_{i=1}^{p-1} \frac{s_{i}(X, Y)}{i}$$
	for $X$ and $Y$ in the Lie algebra, and $s_{i}(X, Y)$ being the coefficient of $t^{i-1}$ in the formal expression $\operatorname{ad}(t X+Y)^{p-1}(X)$.  Since $p > 2$ the polynomial $s_{i}(X, Y) = 0$ for every $i$ as desired. It remains to check that every extension (\ref{extlie}) of restricted Lie algerbas that splits over every factor $\Lie(\alpha_p) \subset \Lie(\alpha)$ arises this way. As observes above, the restricted power operation on $\cL$ is additive. Now consider the subspace $V$ of $ \cL$ consisting of elements annulated by $[p]$-power operation. The projection defines an embedding $V \mono \Lie(\alpha)$. Since the extension has a section over each $\Lie(\alpha_p)$ the embedding is an isomorphism and we win.
\end{proof}

\subsection{Geometric description of $\nG_0$.}\label{proof2}
Let
\begin{equation}\label{ext3}
1\to \hat \bW \to \nG_0 \to G_0\to 1
\end{equation}{}
be the extension from Lemma \ref{redtoW}, and let $\widetilde{\alpha}$ be its restriction to $\alpha  $.

 Let us check that $\widetilde{\alpha}$ satisfies the assumptions of Corollary \ref{extensions}, that is splits over every subgroup $\alpha_p \subset \alpha$. For any $v \in V$ set $f = \omega_V(v, \cdot) \in V^*$ and define a homorphism  $$\alpha_{p} = \Spec k[\epsilon] / (\epsilon^p) \to \widetilde{G} \subset  \underline{A_h [h^{-1}]^*}, \quad 
\epsilon \mapsto e^{\frac{\epsilon f}{h}}.$$ 
Here $e^{\frac{\epsilon f}{h}}$ denotes the restricted exponent, i.e. $e^{\frac{\epsilon f}{h}} = 1 + \frac{\epsilon f}{h} + \ldots + \frac{1}{(p-1)!}(\frac{\epsilon f}{h})^{p-1}$.
It is easy to see that $e^{\frac{\epsilon f}{h}} \in \underline{A_h [h^{-1}]^*}(\alpha_p)$ normalizes the lattice $A_h \otimes k[\epsilon] / (\epsilon^p)$; therefore it lies in $\widetilde{G}(\alpha_p)$. Then the composition of this homomoprhism with the projection to $\widetilde G_0$ is a lift of the embedding   $\alpha_p \subset \alpha$ coresponding to $v$ as desired.

The natural action of $\Sp_{2n} \subset G_0^0 \subset G_0$ on $\alpha$ by conjugation lifts to an action on $\widetilde \alpha$. Indeed, by Lemma \ref{sectionlemma} the extension (\ref{ext3}) splits uniquely over $\Sp_{2 n} \subset G_0^0$.

For future purposes note that the symplectic basis $(x_i, y_i)$ for $V^*$ gives rise to a scheme-theoretic section of $\pi: \widetilde \alpha \to \alpha$. Namely, $\alpha \cong \alpha_{p}^{2n}$, and we define 
\begin{equation}\label{section}
t : \alpha = \Spec(k[\epsilon_1, \ldots \epsilon_n, \delta_1, \ldots \delta_n,]/(\epsilon_i^p = \delta_i^p = 0)) \to \tilde{\alpha} 
\end{equation}
by $e^{\frac{\epsilon_1 x_1}{h}} \ldots e^{\frac{\epsilon_n x_n}{h}} e^{\frac{\delta_1 y_1}{h}} \ldots e^{\frac{\delta_n y_n}{h}}$. The section $t$ is not a group homomoprhism and it does depend on the choice of symplectic basis. However, its differential 
$$dt(e) : \Lie(\alpha) \to \Lie(\widetilde{\alpha})$$
is the unique linear map compatible with the restricted power operation.

Recall that  a connection on a (trivial) $\hat \bW$-torsor $\widetilde{\alpha}$  is a function
\begin{equation}
\nabla : \{ \text{sections } s: \alpha \rightarrow \widetilde{\alpha} \} \rightarrow \Omega_{\alpha}^1 \otimes \Lie(\hat \bW) 
\end{equation}
such that, for any $c \in \hat \bW(\alpha) $, one has $\nabla(cs) = \nabla(s) + c^{-1}dc$. 	Denote by $\conn(\widetilde{\alpha}, \hat \bW)$ the set of $\hat \bW$-connections on $\widetilde{\alpha}$.
More generally, we define the space $\underline{\conn}(\widetilde{\alpha}, \hat \bW)$ of connections on $\widetilde{\alpha}$ to be the functor $(k-\text{algebras})^{op}\to \text{Sets}$ sending an algebra $R$ to
the set of functions 
\begin{equation}
\nabla: \{ \text{sections } s: \alpha \times \Spec R \rightarrow \widetilde{\alpha} \} \rightarrow \Omega_{\alpha}^1 \otimes \Lie(\hat \bW) \otimes R
\end{equation}
with  $\nabla(cs) = \nabla(s) + c^{-1}dc$, for every $c\in \bW( \alpha \times \Spec R )$. 
The group scheme $S_{\alpha}$  of automorphisms of the scheme $\alpha$ acts on the space   $\underline{\conn}(\widetilde{\alpha}, \hat \bW)$. In particular, for any subgroup $H\subset S_{\alpha}$, we have 
a subset $\conn(\widetilde{\alpha}, \hat \bW)^{H}\subset \conn(\widetilde{\alpha}, \hat \bW)$ of $H$-invariant connections. 
\begin{lm}\label{l1}
	There exists a unique $ \Sp_{2n} \ltimes \alpha$-invariant $\hat \bW$-connection $\nabla$ on $\widetilde{\alpha}$.  
\end{lm}

\begin{proof}
	Denote by $\conn(\widetilde{\alpha}, \hat \bW)^{\alpha}$ the set of $\alpha$-invariant connections. We have that $$\conn(\widetilde{\alpha}, \hat \bW)^{\alpha} = \{\text{linear maps }f: \Lie(\widetilde{\alpha}) \rightarrow \Lie(\hat \bW) \text{ such that } f \vert_{\Lie(\hat \bW)} = \Id\}.$$
For an $\alpha$-invariant connection $\nabla$,  the corresponding $f$ is given by the formula 
$$ (\nabla(s) - ds)  \circ d\pi(e) + \Id,$$ for any section $s$. 
	
	Now since $\Sp_{2n}$ normalizes $\alpha$ in $G_0^0$ we get that $\Sp_{2n}$ acts on $\conn(\widetilde{\alpha}, \hat \bW)^{\alpha}$. Suppose we are given two $\Sp_{2n}$-invariant connections in $\conn(\widetilde{\alpha}, \hat \bW)^{\alpha}$. Then their difference gives a morphism $\Lie(\alpha) \rightarrow \Lie(\hat \bW)$ of  representation of $Sp_{2n}$, which has to be trivial since  $\Lie(\alpha)$ is a non-trivial irreducible representation whereas the action of  $Sp_{2n}$ on $\Lie(\hat \bW)$ is trivial.  Thus we get the uniqueness.

	To prove the existence take (a unique) $f : \Lie(\tilde \alpha) \rightarrow \Lie(\hat \bW)$ that commutes with the restricted power operation (see the proof of Corollary \ref{extensions}). This moprhism is $\Sp_{2n}$-invariant since the action of $\Sp_{2n}$ respects the restricted structure. 
	
\end{proof}

	We will need an explicit formula for the $\Sp_{2n} \ltimes \alpha$-invariant connection $\nabla$. Let $t$ be the section defined in (\ref{section}). We claim 
	that
	\begin{equation}\label{explicitformula}
	\nabla(t) = \frac{\eta}{h} = \Sigma \frac{\delta_id\epsilon_i}{h}
	\end{equation}	
	To see this let us show that the connection given by (\ref{explicitformula}) is $\Sp_{2n} \ltimes \alpha$-invariant. Pick a $k$-algebra $R$ and a point $a \in \alpha(R)$ given by $\epsilon_i \mapsto \epsilon_i^{\prime} \in R$. Then $t(a)$ is an $R$-point of $\widetilde{\alpha}$ that acts on $\widetilde{\alpha}$ by translation $\gamma$. The composition $\gamma \circ t : \alpha \times \Spec(R) \to \widetilde{\alpha} \times \Spec(R)$ is given by the formula   
		\begin{equation}\label{a}
	t(a)
	\Pi e^{\frac{\epsilon_ix_i}{h}} 
	\Pi e^{\frac{\delta_iy_i}{h}} = 
	e^{-\frac{\Sigma \epsilon_i\delta_i^{\prime}}{h}}
	\Pi e^{\frac{(\epsilon_i + \epsilon_i^{\prime})x_i}{h}} 
	\Pi e^{\frac{(\delta_i + \delta_i^{\prime})y_i}{h}}. 
	\end{equation} 
	Let $\bar{\gamma}$ be the translation by $a$ acting on $\alpha$. Then $t_{\gamma} = \gamma^{-1}t\bar{\gamma}$ defines another section of $\pi$. 
	The	invariance of $\nabla$ under the action of $\gamma$ reads as
	\[
	\bar{\gamma}^*\nabla(t) = \nabla(t_{\gamma}).
	\]

	Since $\bar{\gamma}(\epsilon_j) = \epsilon_j + \epsilon_j^{\prime}$ and $\bar{\gamma}(\delta_j) = \delta_j + \delta_j^{\prime}$, we have that 
	\[
	\bar{\gamma}^*\nabla(t) = 	\bar{\gamma}^*\frac{\eta}{h} = \frac{\eta}{h} + \Sigma \frac{\delta_i^{\prime}d\epsilon_i}{h}
	\]
	On the other hand, from (\ref{a}) we have $t_{\gamma} = e^{\frac{\Sigma \epsilon_i\delta_i^{\prime}}{h}}t$, and therefore
	\[
	\nabla(t_{\gamma}) = \nabla(t) + e^{-\frac{\Sigma \epsilon_i\delta_i^{\prime}}{h}}de^{\frac{\Sigma \epsilon_i\delta_i^{\prime}}{h}} = \frac{\eta}{h} + \Sigma \frac{\delta_i^{\prime}d\epsilon_i}{h}.
	\]
	Thus $\nabla$ is $\alpha$-invariant. Let us show that $\nabla$ is also $\Sp_{2n}$-invariant. Indeed, the morphism $\{f: \Lie(\widetilde{\alpha}) \rightarrow \Lie(\hat \bW)\}$ coincides with the differential of $t$, which is, as we observed above, a unique linear map compatible with the restricted power operation. Therefore it is  $\Sp_{2n}$-invariant.

  Define $S_{\alpha}$ to be the group scheme of automorphisms of the scheme $\alpha$, that is $$S_{\alpha}(T) = \Aut_T(\alpha \times T).$$ Define also $S_{\widetilde{\alpha}}^{\hat \bW}$ to be the {\it fpqc} sheaf of automorphisms of the torsor $\widetilde{\alpha}$, that is
   $$S_{\widetilde{\alpha}}^{\hat \bW}(T) = \{ \phi \in S_{\alpha}(T), \tilde \phi : \widetilde{\alpha} \times T \rightarrow \phi^*(\widetilde{\alpha} \times T)\}.$$
   Finally, define $S_{\widetilde{\alpha}}^{\nabla}$ to be subsheaf of $S_{\widetilde{\alpha}}^{\hat \bW}$ of endomorphisms preserving the connection $\nabla$ on $\widetilde{\alpha}$.
 \begin{lm}\label{l2}
 	The morphism $ S_{\widetilde{\alpha}}^{\hat \bW} \rightarrow S_{\alpha}$ fits into a short exact sequence
 	\[
 		1\to \Mor(\alpha, \hat \bW) \rightarrow S_{\widetilde{\alpha}}^{\hat \bW} \rightarrow S_{\alpha} \to 1.
 	\]
 	Moreover, the kernel of the composition $S_{\widetilde{\alpha}}^{\nabla} \mono  S_{\widetilde{\alpha}}^{\hat \bW} \to S_{\alpha}$ is the sheaf $\hat \bW \subset \Mor(\alpha, \hat \bW)$ of constant maps. Finally, the image of  $S_{\widetilde{\alpha}}^{\nabla}$ in $ S_{\alpha}$ belongs to $G_0$.

 \end{lm}

\begin{proof}
The map $\Mor(\alpha, \hat \bW) \rightarrow S_{\widetilde{\alpha}}$ takes $f$ to the translation by $f \circ \pi$, and exactness is immediate because $\widetilde \alpha$ is a trivial torsor. If the translation by $f \circ \pi$ preserves the connection, then $df = 0$. This implies that $f$ is constant, that  is $f \in \hat \bW$. 
	
	To see that $\im(S_{\widetilde{\alpha}}^{\nabla}) \subset G_0$ pick $\gamma \in S_{\widetilde{\alpha}}^{\nabla}$ and let $\bar \gamma$ be its image in $S_{\alpha}$. We have that
	\begin{equation}\label{eq23}
	\bar{\gamma}^*\nabla(t) = \nabla(t_{\gamma}) = \nabla(t) + c^{-1}dc \in \Omega_{\alpha}^1 \otimes \Lie(\hat \bW)
	\end{equation}
	for some  $c= (1 + \Sigma_{i>0} a_i h^{-i})\in  \hat \bW(\alpha)$.
	Consider the group homomorphism $\hat \bW \rightarrow \mathbb{G}_a$ that takes a series 
	$(1 + \Sigma a_i h^{-i}) \in \hat \bW(\alpha)$ to $a_1$. This defines a morphism $\Omega_{\alpha}^1 \otimes \Lie(\hat \bW) \rightarrow \Omega_{\alpha}^1$, and the image of $\nabla(t) \in \Omega_{\alpha}^1 \otimes \Lie(\hat \bW)$ is precisely $\eta$. Hence, we have from  (\ref{eq23})
	$$\bar{\gamma}^*\eta =  \eta + da_1 $$
	as desired.
		
\end{proof}

Recall from Lemma \ref{sectionlemma} that the extension $\nG_0 \to G_0$ admits a unique splitting $G^0_0 \mono \nG_0$ over  $G^0_0$.
The left action of $\nG_0$ on   $\nG_0/G^0_0 \iso \widetilde{\alpha}$ defines a homomorphism
\begin{equation}\label{eq11}
\nG_0 \to S_{\widetilde{\alpha}}^{\hat \bW}.
\end{equation} 
\begin{Th}\label{widetildeG_0}
Let $\nabla$ be a  $\Sp_{2n} \ltimes \alpha$-invariant connection on $\tilde{\alpha}$.  Then homomorphism (\ref{eq11}) induces an isomorphism $\nG_0 \iso S_{\widetilde{\alpha}} ^\nabla $.
	
\end{Th}

\begin{proof}
	Let us show that $\nG_0$ is a subsheaf of $S_{\widetilde{\alpha}}^{\nabla}$, that  is, $\nabla$ is $G_0$-invariant. 
	Denote by $H$ the commutator subgroup of $G_0$. We shall first show that $\nabla$ is $H$-invariant.

	 By Lemma \ref{H} the group $H$ is generated by $\alpha$, $\Sp(2n)$,  and a certain one-parameter subgroup $\lambda (\tau): \bG_a \to G^0_0$. By assumption, $\nabla$ is $\Sp_{2n} \ltimes \alpha$-invariant. It remains to check
	 that $\nabla$ is $\bG_a$-invariant.  We use formula  (\ref{explicitformula}) describing $\nabla $ in coordinates  corresponding to the trivialization $t$ of the torsor $\tilde \alpha$.  
	  First, assume that $p>3$. Homomorphism $\lambda$ has a unique lifting $\tilde \lambda $ to  $\widetilde G^{0,e}_0$ that can be explicitly computed using 
	the construction from Lemma  \ref{1-subgroups} 
	 $$\tilde \lambda  =  e^{\frac{\tau x_1^3}{h}}:  \bG_a \to \widetilde G^{0,e}_0.$$  
	The invariance of $\nabla$ under the action of $\bG_a$ reads as
\begin{equation}\label{invariancegamma}
\gamma^*\nabla(t) = \nabla(t'),
\end{equation}
where $t': \bG_a \times \alpha \to \widetilde \alpha $ is the composition 
$$ \bG_a \times \alpha  \rar{ \Id \times \gamma} \bG_a \times \alpha \rar{\Id \times t} \bG_a \times \widetilde \alpha \rar{\Id \times \tilde \gamma ^{-1}}  \bG_a \times \widetilde \alpha \rar{\proj_{\widetilde \alpha} } \widetilde \alpha .$$
We have to compute $t'$. The following equality of morphisms $\bG_a \times \alpha \to \nG_0$  holds.
\begin{equation}\label{x_i}
e^{\tau\frac{x_1^3}{h}} \Pi e^{\frac{\epsilon_ix_i}{h}}e^{\frac{\delta_iy_i}{h}} = e^{-\tau\frac{2\delta_1^3}{h}} (e^{\frac{(\epsilon_1+3\tau\delta_1^2)x_1}{h}}   e^{\frac{\delta_1y_1}{h}} \ldots e^{\frac{\epsilon_nx_n}{h}}e^{\frac{\delta_ny_n}{h}})e^{\frac{\tau((x_1 + \delta_1)^3 -3\delta_1^2x_1 - \delta_1^3)}{h}},
\end{equation}
We claim that the last factor 
$e^{\frac{\tau((x_1+ \delta_1)^3 -3\delta_1^2x_1 - \delta_1^3)}{h}} $ maps   $\bG_a \times \alpha $   to $ G_0^0\subset \nG_0$.
\[
\begin{tikzcd}
G_0^0  \arrow[d, "Id"] \arrow[r, ""]  &  \widetilde G^{0,e}_0 \arrow[d, ""]  \\
G_0^0   \arrow[r, "Id"]  & G_0^0
\end{tikzcd}
\]
Indeed, the same formula  defines an extension of the morphism  $\bG_a \times \alpha \to \nG_0$   to a morphism  from a  {\it reduced} scheme $\bG_a \times \Spec k[[\epsilon_i, \delta_j]]$ to $\nG_0$\footnote{Indeed,
the $p$-th power of $(x_1 + \delta_1)^3 -3\delta_1^2x_1 - \delta_1^3 = x_1^3 +3 x_1^2 \delta_1 \in A_h[[\delta_1]]$ is zero.}. The composition 
of the latter with the projection $\nG_0 \to G_0$ lands in $G^0_0 \subset G_0$. But the projection $ \widetilde G^{0,e}_0 \to G^0_0 $  induces an isomorphism on points with values in any reduced $k$-algebra.
Thus, the morphism $ e^{\frac{\tau((x_1 + \delta_1)^3 -3\delta_1^2x_1 - \delta_1^3)}{h}}: \bG_a \times \Spec k[[\epsilon_i, \delta_j]] \to \nG_0$ factors through $ G_0^0$ and the claim follows.
\[
\gamma^* \nabla(t) = \gamma^* \frac{\eta}{h} = \frac{\eta}{h} + \frac{\delta_1d(3\tau\delta_1^2)}{h} =  \frac{\eta}{h} + \frac{2\tau d\delta_1^3}{h},
\]
\[
\nabla(t') = \nabla(t) + e^{-\tau\frac{2\delta_1^3}{h}}de^{\tau\frac{2\delta_1^3}{h}} = \frac{\eta}{h} + \frac{2\tau d\delta_1^3}{h}.
\]

For $p = 3$, the lift $\tilde \gamma $ is given by   $e^{\tau\frac{x_1^2y_1}{h}}$. Write
\begin{equation}\label{char3}
e^{\tau\frac{x_1^2y_1}{h}} \Pi e^{\frac{\epsilon_ix_i}{h}}e^{\frac{\delta_iy_i}{h}} = f(\tau, \epsilon_1, \delta_1) (e^{\frac{(\epsilon_1+\tau\delta_1\epsilon_1)x_1}{h}}e^{\frac{(\delta_1+\tau\delta_1^2)y_1}{h}} e^{\frac{\epsilon_2x_2}{h}}e^{\frac{\delta_2y_2}{h}} 
\ldots e^{\frac{\epsilon_nx_n}{h}}e^{\frac{\delta_ny_n}{h}})s
\end{equation}
for some uniquely determined $s \in G_0^0(\bG_a \times \alpha) \subset \widetilde G^{0,e}_0(\bG_a \times \alpha)$ and $f(\tau, \epsilon_1, \delta_1) \in \hat \bW(\bG_a \times \alpha)$.
We claim that 
\begin{equation}\label{char3bis}
f(\tau, \epsilon_1, \delta_1) = e^{\tau\frac{\delta_1^2\epsilon_1}{h}}. 
\end{equation}

A direct verification of this formula is unpleasant; instead 
we deduce it from the following facts. Using a computation in Lie algebras from Lemma \ref{splittings} one verifies (\ref{char3bis})  modulo $\tau^2$. 

Also, it is easy to see that the left hand side is invariant under the action of the multiplicative group
given by $\tau \to \frac{\tau}{a}, \epsilon_1 \to \frac{\epsilon_1}{a}, \delta_1\to a\delta_1, x_1 \to ax_1, y_1 \to \frac{y}{a}$. 
Thus,  the element $f(\tau, \epsilon_1, \delta_1)$   must be also invariant under this transformation.  Finally, $f(\tau, \epsilon_1, \delta_1) $ satisfies the following cocycle condition:
\begin{equation}
 f(\tau_1 + \tau_2, \epsilon_1, \delta_1) = 
 f(\tau_1, \epsilon_1, \delta_1) f (\tau_2, \epsilon_1+ 
 \tau_1 \delta_1\epsilon_1, \delta_1 + \tau_1\delta_1^2).
 \end{equation}
 There exists a unique $f$ satisfying the above properties and it is given by (\ref{char3bis}).
 It follows that $\tilde \gamma$ carries the section $t$ to $t'= e^{\tau\frac{2\delta_i^2\epsilon_i}{h}} t$ and (\ref{invariancegamma}) follows.
  
We have proved that $\nabla$ is $H$-invariant. Note that since $\Sp_{2n} \subset H$ and $\alpha \subset H$ we can see that $\nabla$ is a unique $H$-invariant connection. 

The group scheme $G_0$ acts on $\underline{\conn}(\widetilde{\alpha}, \hat \bW)$, the subgroup $H$ is normal in $G_0$, hence $G_0/H$ acts on the space of $H$-invariant connections. But since the latter consists of one element this action must by trivial. Hence $\nabla$ is $G_0$-invariant.

It follows that  homomorphism (\ref{eq11}) factors through  $S_{\widetilde{\alpha}} ^\nabla  $. Thus, by Lemma \ref{l2} we have a commutative diagram
\[
\begin{tikzcd}
\nG_0  \arrow[d, ""] \arrow[r, "\beta"]  &  S_{\widetilde{\alpha}} ^\nabla   \arrow[d, ""]  \\
G_0  \arrow[r, "Id"]  & G_0
\end{tikzcd}
\]
where  $\beta$ induces an isomorphism on the kernels of the vertical arrows. Hence $\beta$ is an isomorphism as desired.  

\end{proof}

\subsection{Proof of the Basic Lemma.} We will prove the assertion for the extensions by $\hat \bW$ (as opposed to $\Gra_{\bG_m}$) which is a priori stronger than the one stated in \S \ref{planofproof.intro}.

Recall the setup.
 Let $i: V\mono V^{\flat}$ be a morphism of symplectic vector spaces such that the restriction to $V$ of the symplectic form on $V^{\flat}$ is $\omega_V$. Let $\nG_0 \to G_0$ and $\nfG_0 \to G_0^{\flat}$ be the corresponding extensions. Denote by $G_0^{\sharp}\subset G_0^\flat $ the group subscheme that consists of automorphisms preserving  kernel of the homomorphism $i^*: A_0^{\flat} \to A_0$. 
 We have a natural homomorphism $G_0^{\sharp} \to G_0$.
 \begin{Th}[Basic Lemma]\label{lifting} 
 The homomorphism $G_0^{\sharp} \to G_0$ lifts uniquely to a homomorphism of central extensions
 	$$ \nfG_0  \times _{G_0^\flat}  G_0^{\sharp} \to \nG_0.$$
\end{Th}
 \begin{proof} The uniqueness follows from Corollary \ref{hom}. Let us prove the existence. By Corollary \ref{extensions} the morphism $i: \alpha \rightarrow \alpha^{\flat}$ lifts uniquely to a morphism $\tilde i: \widetilde{\alpha} \rightarrow \widetilde{\alpha}^{\flat}$ of extensions. Moreover,
 the pullback of the (unique) $\Sp (V^\flat) \ltimes \alpha ^\flat$-invariant connection $\nabla^\flat $ on $\widetilde{\alpha}^{\flat}$ is the (unique) $\Sp (V) \ltimes \alpha $-invariant connection $\nabla$  on $\widetilde{\alpha}$.
 Thus, we have a homomorphism 
 $$ S_{\widetilde{\alpha}^\flat} ^{\nabla ^\flat}  \times _{G_0^\flat}  G_0^{\sharp} \to S_{\widetilde{\alpha}} ^\nabla $$
 lifting $G_0^{\sharp} \to G_0$. It remains to apply
 Theorem \ref{widetildeG_0}.
 \end{proof}
 \begin{cor}\label{refofbasicl} There exists a unique isomorphism of central extensions of $G_0^{\sharp}$:
 $$ \nfG_0  \times _{G_0^\flat}  G_0^{\sharp} \iso \nG_0 \times _{G_0} G_0^{\sharp}.$$
 \end{cor}

\subsection{Proof of the Proposition \ref{keyprop}.}
We will prove a stronger assertion for the extensions by $\hat \bW$ (as opposed to $\Gra_{\bG_m}$).
Let $(V, \omega)$ be a symplectic vector space and let $\eta$ be a homogeneous $1$-form on the scheme $\bf V$ whose differential equals $\omega$, that is  a vector $\eta\in V^*\otimes V^*$ whose 
skew-symmetrization is $\omega$. Denote by
$$i: V \mono V^\flat := V \oplus V^*$$
the linear morphism corresponding to the graph $\Gamma_\eta: \bf V \mono \bT^*_ {\bf V}$ of $\eta$. Explicitly, the composition of $i$ with the first projection is $\Id$ and its composition $V\to V^* $ 
with the second projection is given by $\eta\in  V^*\otimes V^*$.
In \S \ref{subsectionsext} we defined a homomorphism $\psi_0: G_0 \to G_0^\flat$. We have to prove that $\psi_0$ lifts uniquely to a homomorphism $\tilde \psi_0: \nG_0 \to \nfG_0$ of extensions.
The uniqueness   follows from Corollary \ref{hom}. To prove the existence we observe that by construction of $\psi_0$ it factors through the subgroup $G_0^{\sharp} \subset  G_0^\flat $  that
consists of automorphisms preserving  kernel of the homomorphism $i^*: A_0^{\flat} \to A_0$ and its composition
$$G_0 \rar{\psi_0 }G_0^{\sharp} \rar{} G_0$$
with restriction morphism is the identity\footnote{Homomorphism $\psi_0$ can be described in a coordinate-free way  as follows. Consider the subgroup $G_0^{\sharp, f}$ of $G_0^{\sharp} $ that consists of scheme-theoretic automorphisms $g$ of $\alpha^\flat$ fitting in the commutative diagram

 $$
\def\normalbaselines{\baselineskip20pt
	\lineskip3pt  \lineskiplimit3pt}
\def\mapright#1{\smash{
		\mathop{\to}\limits^{#1}}}
\def\mapdown#1{\Big\downarrow\rlap
	{$\vcenter{\hbox{$\scriptstyle#1$}}$}}
\begin{matrix}
\alpha^{\flat}  &   \rar{g}    & \alpha^{\flat} \cr
 \mapdown{\pi}  &  &\mapdown{ \pi}   \cr
\alpha   & \rar{\bar{g} }    & \alpha,
\end{matrix}
$$
for some $\bar{g} \in G_0$.
 The restriction of the projection $G_0^{\sharp} \to G_0$ to $G_0^{\sharp, f}$ is an isomorphism and $\psi_0$ is its inverse.}.
 Consider the homomorphism $$\nG_0 \rar{(\Id, \psi_0)} \nG_0 \times _{G_0} G_0^{\sharp} .$$ Using Corollary \ref{refofbasicl} we get a morphism 
 $$\nG_0 \times _{G_0} G_0^{\sharp}  \iso  \nfG_0  \times _{G_0^\flat}  G_0^{\sharp} \rar{\proj} \nfG_0.$$
Its composition with  $(\Id, \psi_0)$ is the desired lift $\tilde \psi_0: \nG_0 \to \nfG_0$.

\section{$\bG_m$-equivariant quantizations}\label{ktconjecture}
In this section we consider quantizations of symplectic varieties $(X, \omega)$ equipped with an action of the multiplicative group $\bG_m$  such that the form $\omega$ has a positive weight  $m$ with respect to this action and $m$ is invertible in $k$.
We recall the notion of a $\bG_m$-equivariant Frobenius constant quantization $O_h$  of such $(X, \omega)$. By definition,  $O_h$ is a $\bG_m$-equivariant sheaf of  $\cO_{X'[h]}$-algebras on $X' \times \Spec k[h]$. In particular, specializing $h=1$ we have a sheaf  $O_{h=1}$ of $\cO_{X'}$-algebras over $X'$. We show that if the action of $\bG_m$ on $X$ is contracting  then   $O_{h=1}$ is an Azumaya algebra
over $X'$ and using Theorem \ref{main} compute its class in the Brauer group $\Br(X')$ proving a conjecture of  Kubrak and Travkin \cite{kt}.

\subsection{Definition of $\bG_m$-equivariant quantizations.}
Let $X$ be a smooth variety over $k$ equipped with a symplectic $2$-form $\omega$
and a $\bG_m$-action 
\begin{equation}\label{multgraction}
\lambda: \bG_m \times X \to X.
\end{equation}
We shall say that $\omega$ is of weight  $m$ with respect to the $\bG_m$-action 
if the following identity holds 
in $\Gamma(\bG_m \times X, \Omega^2_{ \bG_m \times X/\bG_m})$
\begin{equation}\label{multgractionomega}
\lambda^* \omega  =  z^m \proj_X^* \omega . 
\end{equation}
Here $z$ denotes the coordinate   on $\bG_m$ and $\proj_X: \bG_m \times X \to X$ the projection.
For the duration  of this section we shall assume that $\omega$ is of weight $m$ with 
$m$ invertible in $k$. 

The $\bG_m$-action  on $X$ defines a homomorphism from the Lie algebra of $\bG_m$ to the Lie algebra of vector fields on $X$. Denote by $\theta$ the image of the generator of $\Lie \bG_m$. 
Formula (\ref{multgractionomega}) together with the identity $d\omega=0$  imply
that 
$$d \iota_{\theta} \omega = m \omega. $$
Hence, setting $\eta = \frac{1}{m} \iota_{\theta} \omega $, defines a restricted Poisson structure on $X$. 
Endow  $X'[h]\colon= X' \times \Spec k[h]$ with the $\bG_m$-action given by the composition $$\bG_m \times X'  \rar{F\times \Id} \bG_m \times X'  \rar{\lambda} X'$$ 
(where $F: \bG_m \to \bG_m$ is given by $F^*(z) =z^p$) on the first factor and by $h\mapsto z^m h$ on the second factor.

A $\bG_m$-equivariant Frobenius-constant quantization  of $X$ consists of a $\bG_m$-equivariant sheaf $O_h$ of associative $\cO_{X'[h]}$-algebras on $X'[h]$ together with an isomorphism of $\bG_m$-equivariant 
$\cO_{X'}$-algebras
\begin{equation}\label{gmmodh}
O_h/(h) \iso \cO_X
\end{equation}
such that
$O_h$ is locally free as an $\cO_{X'[h]}$-module and the restriction $\cO_h:= \lim O_h/(h^n)$ 
of $O_h$ to the formal completion of $X'[h]$ along the divisor $h=0$ (equipped with the central homomorphism $s:  \cO_{X'}[[h]] \to \cO_h$ and  (\ref{gmmodh})) is a Frobenius-constant quantization  of  $X$
compatible with the restricted Poisson structure given by the $1$-form  $\eta = \frac{1}{m} \iota_{\theta} \omega $. 

For example, if $X$ is affine, then a $\bG_m$-equivariant Frobenius-constant quantization  of $X$ is determined by a graded $\cO(X')[h]$-algebra $O_h(X'[h])$ (with $\deg h=m$) together with  $O_h(X'[h])/(h)\iso \cO(X)$.

\subsection{$\bA^1$-action.} Below we shall consider $\bG_m$-actions on a scheme $X$ satisfying the following property: morphism  (\ref{multgractionomega}) extends to a morphism 
\begin{equation}\label{multgractionext}
\tilde \lambda: \bA^1 \times X \to X.
\end{equation}
If $X$ is reduced and separated, which we shall assume to be the case for rest of this section,  then $\tilde \lambda $ defines an action of the monoid $\bA^1$ on $X$. 
In particular, the restriction of $\tilde \lambda$ to the closed subscheme $X \mono \bA^1 \times X$ given by the equation $z=0$ factors through the subscheme $X^{\bG_m} \mono X$ of fixed points: 
$$\tilde \lambda_0: X \to X^\bG_m \mono X.$$
Moreover, $\tilde \lambda $ exhibits $X^{\bG_m}$ as a $\bA^1$-homotopy retract of $X$. 

Also note if $X$  is a proper scheme with a nontrivial action of $\bG_m$ then  $\tilde \lambda$ does not exist. This  can be seen by looking at  the closure of a $1$-dimensional $\bG_m$-orbit in $X$.

\subsection{Main result.}
By definition a $\bG_m$-equivariant Frobenius-constant quantization  $O_h$ gives rise to  a  Frobenius-constant quantization  $\cO_h$ and thus a class $\rho(O_h)\in H^1_{et}(X',  \cO_{X'}^*/\cO_{X'}^{* p} )$ (see \S \ref{bktheorem}). Denote by $[\rho(O_h)] \in \Br (X')$ the image of $\rho(O_h)$ under the homomorphism $H^1_{et}(X',  \cO_{X'}^*/\cO_{X'}^{* p} )\to H^2_{et}(X',  \cO_{X'}^*)$.

Recall from (\cite{ov}, \S 4.2) a homomorphism 
\begin{equation}\label{milne}
\Omega^1(X')  \to \Br (X'), \quad \eta \mapsto [\eta]
\end{equation}
that carries a $1$-form $\eta$ to the class of the Azumaya algebra $D_X$ restricted to the graph of the embedding $\Gamma_\eta: X' \to \bT^*_{X'}$ given by $\eta$.  
Finally, denote by $O_{h=1}$  the sheaf of $\cO_{X'}$-algebras  $O_h/(h-1)$. The following result has been conjectured in \cite[\S 0.3, Question 2]{kt}.

\begin{Th} Let $X$ be a smooth variety over $k$ equipped equipped with a $\bG_m$-action (\ref{multgractionomega}) and a  symplectic form $\omega$ of weight $m>0$  relatively prime to the characteristic of $k$.
	Assume that the morphism  (\ref{multgraction}) extends to a morphism (\ref{multgractionext}).
	Then, for every $\bG_m$-equivariant Frobenius-constant quantization $O_h$ of  $X$, the restriction of $O_h$ to 
	$X'[h, h^{-1}] = X' \times \Spec k[h, h^{-1}] \subset X'[h]$
	is an Azumaya algebra. Moreover, the following equality in $\Br(X')$ holds
	$$[O_{h=1}] = [\frac{1}{m} \iota_{\theta} \omega]  + \tilde \lambda_0^*[\rho(O_h)] .$$
\end{Th} 
\begin{proof}
	To prove the Azumaya property of $O_h(h^{-1})$ consider the morphism
	$$ \psi: O_h \otimes_{\cO_{X'[h]}} O_h^{op} \to \End_ {\cO_{X'[h]}}(O_h).$$
	This is a morphism of vector bundles over $X'[h]$ of the same rank. We have to prove that $\psi$ is an isomorphism away from the divisor $h=0$. Denote by 
	$$\det \psi: \wedge^{top} (O_h \otimes_{\cO_{X'[h]}} O_h^{op}) \to \wedge^{top}(\End_ {\cO_{X'[h]}}(O_h))$$
	the determinant of $\psi$ and by $Z\mono X'[h]$.  Let $(x_0, h_0) \in Z$ be a $k$-point of $Z$. We shall check that $h_0=0$. 
	Using (\ref{multgractionext})  and the positivity of $m$ the $\bG_m$-action on $X'[h]$ extends to a morphism 
	\begin{equation}\label{multgractionexth}
	\tilde \lambda(h) : \bA^1 \times X'[h] \to X'[h].
	\end{equation}
	It follows that the closure $T$ of the $\bG_m$-orbit of $(x_0, h_0)$ intersects the divisor $h=0$ at some point $(x_0', 0)$. Since $Z$ is closed and $\bG_m$-invariant we have that $T\subset Z$ {\it i.e.}, 
	$\det \psi$ is identically $0$ on $T$. On the other hand, using the Azumaya property of the formal quantization $\cO_h(h^{-1})$ we see that the restriction of $\psi$  to the formal punctured neighborhood
	of $(x_0', 0)\in T$ is an isomorphism. This contradiction proves the first assertion of the Theorem.

	For the second one, 
	consider algebra  $D_{X, h}$  obtained from the filtered algebra of differential $D_{Y}$ operators via the Rees construction (see \S 
	\ref{intro.dif}). The $p$-curvature homomorphism makes $D_{X, h}$ into an algebra over $S^\cdot T_{X'}[h]$.
	The graph $\Gamma_{\eta}: X' \rightarrow \mathbb{T}_{X'}^*$ of the differential form $\eta= \frac{1}{m} \iota_{\theta} \omega$ defines a sheaf of ideals $I_{\Gamma_{\eta}} \subset S^{\cdot}T_{X'}$.
	The quotient $D_{X, [\eta], h}=D_{X, h} /I_{\Gamma_{\eta}} $ can be viewed as a $\bG_m$-equivariant sheaf of $\cO_{X'[h]}$-algebras over $X'[h]$. By construction, the restriction of $D_{X, [\eta], h}$ 
	to the formal completion of $X'[h]$ along the divisor $h=0$ is isomorphic to the algebra $\cD_{X, [\eta], h}$ constructed in \S \ref{intro.centralred}.
	Now given a $\bG_m$-equivariant Frobenius-constant quantization $O_h$ 
	we consider the tensor product  $O_h \otimes_{\cO_{X'}[h]} D_{X, [\eta], h}^{op}$. Using Theorem \ref{main} and  the Beauville-Laszlo theorem (\cite{bela}) there exists a sheaf $O^\sharp_h$ of  $\cO_{X'[h]}$-algebras over $X'[h]$
	whose restriction to $X'[h, h^{-1}]$ is $(O_h \otimes_{\cO_{X'}[h]} D_{X, [\eta], h}^{op})(h^{-1})$ and whose restriction  to the formal completion of $X'[h]$ along the divisor $h=0$ is an Azumaya algebra. 
	It follows that $O^\sharp_h$ is an Azumaya algebra over $X'[h]$.
	We claim that the following equality holds in $\Br(\bA^1 \times X'[h])$
	\begin{equation}\label{keyequiv}
	\tilde \lambda(h)^*([O^\sharp_h]) = \proj ^*_{X'[h]}([O^\sharp_h]) .
	\end{equation}
	Indeed, since $(O_h \otimes_{\cO_{X'}[h]} D_{X, [\eta], h}^{op})(h^{-1})$ is $\bG_m$-equivariant  the equality holds after the restriction to  $\bG_m \times X'[h, h^{-1}]$. Now the claim follows from the injectivity of the restriction morphism  $\Br(\bA^1 \times X'[h]) \to \Br(\bG_m \times X'[h, h^{-1}])$.  Restricting the classes in (\ref{keyequiv}) to the divisor $ X'[h] \rar{z=0}  \bA^1 \times X'[h]$ we find that  
	\begin{equation}\label{keyequivbis}
	\tilde \lambda(h)^*_0([O^\sharp_h]) = [O^\sharp_h].
	\end{equation}
	Morphism $\tilde \lambda(h)^*_0: X'[h] \to X'[h]$ factors as follows 
	$$X'[h] \rar{\proj_{X'}} X' \rar{\tilde \lambda(h)^*_0} X'  \rar{h=0} X'[h] .$$
	By Theorem \ref{main} the restriction of $[O^\sharp_h]$ to the divisor  $X'  \rar{h=0} X'[h] $ is equal to $[\rho(O_h)]$. Using (\ref{keyequivbis}) and restricting to the divisor $h=1$ we find that 
	$$[O^\sharp_{h=1}]= [\rho(O_h)]$$
	as desired.

\end{proof}

\section{Appendix}\label{appendix}
\subsection{The affine grassmannian for $\bG_m$.}
Let $G$ be an algebraic group over a field $k$. Denote by 
$$LG: \text{Aff}_k^{op} \to \text{Groups}$$ 
the corresponding loop group, that is a sheaf of groups on the category of affine schemes over $k$ equipped with the {\it fpqc} topology sending $k$-scheme $\Spec R$ to 
$G(R((h)))$. Also, let 
$$L^+G: \text{Aff}_k^{op} \to \text{Groups}$$ 
be the sheaf of groups  sending $k$-scheme $\Spec R$ to 
$G(R[[h]])$. It is known ({\it e.g.} see  \cite[Proposition 1.3.2]{zhu} ) that $L^+G$ is represented by a  group scheme over $k$ and that 
$LG$ is an ind-affine scheme.  Denote by $\Gra _G$ the affine grassmannian for $G$. By definition, $\Gra _G$ is the {\it fpqc}  sheaf associated to the presheaf $R \mapsto LG(R)/L^+G(R)$.

Recall the structure of the affine grassmannian for $\bG_m$. The following result is well-known (see, for example, \cite{cc}); for the reader's convenience we include its proof.
\begin{lm}\label{decomp}
For a commutative ring $R$ such that $\Spec R$ is connected there is a decomposition
$$R((h))^* = R^* \times \bW(R) \times \mathbb{Z} \times \hat \bW (R),$$
where $\bW(R)$ is the subgroup of $R[[h]]^*$ formed by formal power series with constant term $1$,
$\hat \bW (R)$  is the group of polynomials  of the form $ 1 + \Sigma a_i h^{-i}$  with nilpotent coefficients $a_i \in R$. In addition, we have that
  $$\hat \bW (R) = \ker(R[h^{-1}]^* \rar{f\mapsto f(0)} R^*).$$
\end{lm}{}

\begin{proof}
The claim follows from the fact that under the assumptions of the lemma
$$R((h))^{*}=\left\{\sum_{i} a_{i} h^{i} \in R((h)) ; \exists i_{0}: a_{i_{0}} \in R^{*}, a_{j} \text { nilpotent for all } j<i_{0}\right\}.$$

To show this replace $R$ by $R/ \mathfrak{N}_{R}$, where $\mathfrak{N}_{R}$ stands for the nilradical. We need to show that a Laurent polynomial is invertible if and only if its first nonzero coefficient is invertible in $R$. Suppose that 
\begin{equation}\label{thelene}
A(h)  B(h) = 1
\end{equation}
 for $A(h), B(h) \in R((h))$ such that 
\begin{equation}
    \begin{split}
        A(h) = a_{-N} h^{-N} + a_{-N+1} h^{-N+1} + ... \\
        B(h) = b_{-N} h^{-M} + b_{-M+1} h^{-M+1} + ...
    \end{split}{}
\end{equation}{}
where $b_{-M} \neq 0$ and $a_{-N} \neq 0$. From (\ref{thelene}) we have that $N+M \geq 0$. If $N+M=0$ then $a_{-N} b_{-M} = 1$ and we are done. Otherwise, we have from (\ref{thelene}) 
\begin{equation}\label{eq7}
a_{-N} b_{-M} = 0, \quad a_{-N}b_{-M+1} + a_{-N+1}b_{-M}=0,  \cdots 
\end{equation}
\begin{equation}\label{eq8}
a_{-N}b_{-M + (N+M)} + \ldots + a_{-N+(N+M)}b_{-M} = 1.
\end{equation}
Using (\ref{eq7}) we get $a_{-N}^2b_{-M+1}=0$ and similarly $a_{-N}^ib_{-M+i}=0$ for every $i \leq N+M$.  
Multiplying both sides of (\ref{eq8})  by $a_{-N}^{N+M}$ we infer 
$$a_{-N}^{N+M}(1-a_{-N}b_{-M+(N+M)}) = 0.$$
Since $a_{-N}$ and $1-a_{-N}b_{-M+(N+M)}$ are coprime:  $$R \iso R/(a_{-N}) \times R/(1-a_{-N}b_{-M+(N+M)}),$$
 $\Spec R$  is connected, and  $a_{-N} \neq 0$ we conclude that 
 $1-a_{-N}b_{-M+(N+M)} = 0$.  Hence $a_{-N}$ is invertible as desired.  
\end{proof}{}
Using the Lemma, we have decompositions  
$$ L\bG_m \iso \bG_m \times \bW \times \underline{\bZ} \times \hat \bW, $$
$$ \Gra_{\bG_m} \iso  \underline{\bZ} \times \hat \bW, $$
 where  $\bW$ is the group scheme of big Witt vectors  and $\hat \bW$ is a group ind-scheme whose group of $R$-points
 is defined in the Lemma. 
 
\subsection{Subgroups of $L\GL(n)$.}

\begin{pr}\label{subgroupsofLGL} 
Let $G$ be an affine group scheme over a field $k$,  and let $\phi: G \to LGL(n) $ be a homomorphism. Then there exists an element $g\in GL(n, k((h)))$ such that $\phi $  factors through $gL^+GL(n)g^{-1}$:
$$ G\rar{\phi} g(L^+GL(n)) g^{-1} \mono LGL(n).$$

\end{pr}
\begin{proof}
Set $V=k^n$.
We have to show that there exists a $\phi(G)$-invariant  $k[[h]]$-lattice $$\Lambda \subset V((h)).$$ 
Informally, our $\Lambda$ will be constructed starting with the lattice $\Lambda_0=V[[h]]$  as the intersection $\bigcap_g g \Lambda _0$. Since we make no assumptions on $k$ and $G$ one has give a meaning to the latter. 
We shall do it as follows.

The morphism $\phi$ is given by a matrix $A\in \GL(n, \cO(G)((h)))$ such that
\begin{equation}\label{eqprodmat}
A\otimes A = \Delta (A) \in  \GL(n,  (\cO(G)\otimes \cO(G))((h))),
\end{equation}
where $\Delta:  \cO(G) \to \cO(G)\otimes \cO(G)$ is the comultiplication on $\cO(G)$ given by the product morphism $G\times G \to G$ and such that
the image of $A$ under the evaluation at $1\in G(k)$ homomorphism $\GL(n, \cO(G)((h))) \to \GL(n,k((h)))$ is the identity matrix.

Set
$$\Lambda = \{ v\in   V((h)) \; \text{such\,  that}\;  Av\in V \otimes _k \cO(G)[[h]]\}.$$
Then $\Lambda$ is a $k[[h]]$-submodule of $V[[h]]$ contains $h^N V[[h]]$, for sufficiently large $N$. Hence, $\Lambda$ is a lattice. It remains to show that $\Lambda$ is $\phi(G)$-invariant, that is
$$A( \Lambda)  \subset \Lambda \otimes _{k[[h]]} \cO(G)[[h]].$$

The matrix $A$ defines  $\cO(G)((h))$-linear maps
\begin{equation}\label{eqprodmatcompbis}
\begin{split}
V((h)) \otimes_{k[[h]]}  \cO(G)[[h]] \rar{A\otimes \Id} (V((h)) \otimes_{k[[h]]}     \cO(G)[[h]]) \otimes_ {k[[h]]} \cO(G)[[h]] \to \\
\to V((h)) \otimes_{k[[h]]}     (\cO(G) \otimes_k \cO(G))[[h]],
\end{split}
\end{equation}
where the second map in (\ref{eqprodmatcompbis}) is induced by the embedding 
 \begin{equation}\label{eqcompltensor}
\cO(G)[[h]] \otimes_ {k[[h]]} \cO(G)[[h]]   \to           (\cO(G) \otimes_k \cO(G)   )[[h]]. 
  \end{equation}
  Since the cokernel of (\ref{eqcompltensor}) and 
$ \cO(G)[[h]]$  
are both flat $k[[h]]$-modules it follows that   $\Lambda \otimes _{k[[h]]} \cO(G)[[h]]$ is precisely the preimage  of  $V[[h]] \otimes_{k[[h]]}     (\cO(G) \otimes_k \cO(G))[[h]]$ under the composition
(\ref{eqprodmatcompbis}). 

Hence it suffices to check that  (\ref{eqprodmatcompbis}) carries $A( \Lambda) $ to  $V[[h]] \otimes_{k[[h]]}  (\cO(G) \otimes_k \cO(G))[[h]].$
But the composition 
\begin{equation}\label{eqprodmatcomp}
  V((h))  \rar{A}  V((h)) \otimes_{k[[h]]}  \cO(G)[[h]]  \rar{  (\ref{eqprodmatcompbis})       }    V((h)) \otimes_{k[[h]]}     (\cO(G) \otimes_k \cO(G))[[h]]   
 \end{equation} 
is equal to 
$$ V((h))  \rar{A}  V((h)) \otimes_{k[[h]]}  \cO(G)[[h]]  \rar{  Id \otimes \Delta  }    V((h)) \otimes_{k[[h]]}     (\cO(G) \otimes_k \cO(G))[[h]]   $$
by (\ref{eqprodmat}).   Hence it carries $\Lambda$ to  $V[[h]] \otimes_{k[[h]]}  (\cO(G) \otimes_k \cO(G))[[h]]$
 and we win.
\end{proof}
\begin{cor}\label{hom}
There are no nontrivial homomorphisms from  an affine group scheme to $L\bG_m/L^+\bG_m$. 
\end{cor}
\subsection{Subgroups of $L\PGL(n)$.}
\begin{rem}
The analogues assertion for $L\PGL(n)$ does not hold. 
\end{rem}
Consider the homomorphism of  loop groups {\it fpqc} sheaves
\begin{equation}\label{l.g.ext.}
L\GL(n) \to L\PGL(n)
\end{equation} 
induced by the projection $\GL(n) \to \PGL(n)$. We do not know if  (\ref{l.g.ext.})  is surjective as a morphism of  {\it fpqc} sheaves. However, we shall see below that (\ref{l.g.ext.})  is surjective  over any affine group subscheme of $L\PGL(n)$ of finite type over $k$. For our applications we need a bit more general statement. 

Recall that an affine group scheme $H$ over a perfect field $k$  is said to to be  pro-unipotent if there exists a filtration   $$\cdots \subset H^{\geq i} \subset \cdots  \subset H^{\geq 1}=H$$ 
by normal group subschemes such that 
$$H \iso \lim_{\longleftarrow} H/H^{\geq i}$$
and every quotient $H/H^{\geq i}$ is unipotent ({\it i.e.}, has a finite composition series with all quotient groups isomorphic to the additive group $\bG_a$).

\begin{pr}\label{apploopspgl} 
Let $G$ be an affine group scheme over a perfect field $k$,  and let $\phi: G \to L\PGL(n) $ be a homomorphism. Assume that $G$ has a normal  pro-unipotent group subscheme $G^{\geq 1}\subset G$ such that 
$\phi(G^{\geq 1})\subset L^+\PGL(n) $ and the quotient $G_0=G/G^{\geq 1}$ has finite type over $k$. Then the following assertions hold:
\begin{itemize}
\item[(i)]
The morphism of  {\it fpqc} sheaves  $\tilde G: = G \times _{L\PGL(n)} L\GL(n) \to G$ given by the projection to the first coordinate is surjective for the Zariski topology on $G$ (and, consequently, for the  {\it fpqc} topology).
\item[(ii)]
The following two conditions are equivalent.
\begin{itemize}
\item[(1)]
There exists  an element $g\in \PGL(n, k((h)))$ such that $\phi $  factors through $gL^+\PGL(n)g^{-1}$:
$$ G\rar{\phi} g(L^+PGL(n)) g^{-1} \mono LPGL(n).$$
\item[(2)]
The extension
\begin{equation}\label{fundextension}
1  \to L\bG_m \to \tilde G \to G \to 1
\end{equation}
admits a reduction  $\tilde G ^+$ to $L^+\bG_m$. 
   $$
\def\normalbaselines{\baselineskip20pt
\lineskip3pt  \lineskiplimit3pt}
\def\mapright#1{\smash{
\mathop{\to}\limits^{#1}}}
\def\mapdown#1{\Big\downarrow\rlap
{$\vcenter{\hbox{$\scriptstyle#1$}}$}}
\begin{matrix}
 1 &  \to   & L^+\bG_m  & \to  & \tilde G^+ & \to &  G & \to & 1\cr
&& \mapdown{}  &  &\mapdown{}  &&  \mapdown{\Id} && \cr
1 &  \to   &  L\bG_m   & \to  & \tilde G  & \to &   G & \to & 1
\end{matrix}
 $$
\end{itemize}
\item[(iii)]
Assume that $G_0$ is smooth and connected. Then extension (\ref{fundextension}) 
admits a unique reduction  $\tilde G ^+$ to $L^+\bG_m$. 
\end{itemize}
\end{pr}
\begin{proof} For part (i) observe that the morphism of schemes 
$$L^+\GL(n) \to L^+\PGL(n)$$
admits a section locally for the Zariski topology on $ L^+PGL(n)$. 
Also since every $G^{\geq 1}$-torsor over an affine scheme is trivial the projection
$$ G\epi G_0$$
admits a scheme-theoretic section $s: G_0 \to G$. Hence it suffices to check that the composition 
$G_0\rar{s} G\rar{\phi} L\PGL(n) $ lifts locally for the Zariski topology on $G_0$ to scheme-theoretic morphism $G_0 \to L\GL(n) $.
Set $G_0 = \Spec R$. Then $\phi \circ s $ defines a morphism 
\begin{equation}\label{section2}
\Spec R((h)) \to  \PGL(n).
\end{equation}
The pullback of the $\bG_m$-torsor $\GL(n) \to \PGL(n)$ defines a $\bG_m$-torsor $L$ over  $\Spec R((h))$. Observe that $\phi \circ s $ admits a lifting to  $ L\GL(n) $
if and only if $L$ is trivial. Thus to complete the proof of (i) we have to show that  there exists an affine open covering $\Spec  R= \cup U_i$ such that the pullback of $L$  to $ \Spec \cO(U_i)((h))$ is trivial for every $i$. 
We shall prove  a stronger assertion: the morphism $\Spec R((h)) \to \Spec R$ induces an isomorphism
  \begin{equation}\label{picard}
 \Pic (R) \iso \Pic (R((h))).
  \end{equation}
  Since $G_0$ is a group scheme and $k$ is perfect, the reduction $R_{red}$ is smooth over $k$.  Since $R$ is a finitely generated $k$-algebra, the kernel of the projection $R \to R_{red}$ is a nilpotent ideal. It follows that
   $( R((h))_{red} \iso R_{red}((h))$. Consequently,   we have that
   $$ \Pic (R)  \iso \Pic (R_{red}), \quad  \Pic ( R((h))) \iso  \Pic ( R_{red}((h))).$$
  Next, using regularity of $ R_{red}((h))$ we conclude that $$\Pic (R_{red}((h)))\cong \Class (R_{red}((h))) \cong  \Class (R_{red}[[h]])\cong \Pic (R_{red}[[h]])\cong \Pic (R_{red}).$$ 
  This proves part (i).
  
  For part (ii), let $\tilde G ^+$ be  a reduction of $\tilde G$ to $L^+\bG_m$. Since $L^+\bG_m$ is an affine group scheme (as opposed to merely a group ind-scheme) $\tilde G ^+$  is also an affine group scheme.
  Applying   Proposition \ref{subgroupsofLGL} we conclude that the homomorphism $\tilde G ^+ \to L\GL(n)$ factors through $gL^+\GL(n)g^{-1}$, for some $g\in \GL(n, k((h)))$. Hence $G \to L\PGL(n)$
factors through $gL^+\PGL(n)g^{-1}$. The inverse implication is clear.

  Finally, for part (iii), set  $\bar G:= \tilde G / L^+\bG_m$.  We have to show that the central extension 
  $$\Gra_{\mathbb{G}_m} \to \bar G  \to G$$
  admits a unique splitting. We shall first construct a scheme-theoretic section of the projection $\bar G  \to G$.
  Using  part (i) there exists an open cover $G= \cup U_i$ and sections $s_i: U_i \to \tilde G $ of the projection $\tilde G  \to G$. Let $\bar s_i:  U_i \to \bar G $ be the composition of $s_i$ with the quotient map $\tilde G \to \bar G$. Since $G$ is reduced the morphisms
  $$\bar s_i \bar s_j^{-1}: U_i \cap U_j \to \Gra_{\mathbb{G}_m} = \hat \bW \times \bZ$$ 
  lands at the second factor. Hence the collection $\{\bar s_i \bar s_j^{-1}\}$ defines a \v{C}ech $1$-cocycle for the constant sheaf $\bZ$ on $G$. Since $G$ is irreducible, we have that  $\check{H}^{1}(G, \mathbb{Z})=0$. 
  Thus,  we have a global scheme-theoretic section $\bar s: G \to \bar G$ of the projection $\bar G  \to G$. We claim that every such section satisfying  
   $\bar s(1) =1$
  is a group homomorphism. To see this it suffices to show that the diagram 

   $$
\def\normalbaselines{\baselineskip20pt
	\lineskip3pt  \lineskiplimit3pt}
\def\mapright#1{\smash{
		\mathop{\to}\limits^{#1}}}
\def\mapdown#1{\Big\downarrow\rlap
	{$\vcenter{\hbox{$\scriptstyle#1$}}$}}
\begin{matrix}
G \times G  &  \rar{m}   &  G  \cr
 \mapdown{\bar s \times \bar s}  &  &\mapdown{ \bar s}  \cr
 \bar G \times \bar G &  \rar{\bar m}   &  \bar G  
\end{matrix}
$$
is commutative. In turn, this follows from the fact that every scheme-theoretic morphism from a connected reduced scheme  to $\Gra_{\mathbb{G}_m}$ is constant.
\end{proof}

 \subsection{A representation of $\lsp(2n)$.} In this subsection we prove irreducibility of a certain representation of the Lie algebra $\lsp(2n)$
 that we  used in the proof of Lemma \ref{commutator}. We use notations from   \S \ref{propg0}. 
 
 \begin{lm}\label{irrrepsp}
	For every integer $l$ with $0\leq l < 2(p-1)$, the adjoint representation of the Lie algebra $\lsp(2n) = m^2/m^3$ on $m^l/m^{l+1}$ is irreducible.
\end{lm}
\begin{proof}
Write $m_n^l / m_n^{l+1}$  for $m^k / m^{k+1}$.  It is easy  to  verify the assertion of  the lemma for $n=1$: in fact, the representation of $\lsp(2)= m_1/m_1^2$ on  $m_1^l / m_1^{l+1}$ is  irreducible for every $l\geq 0$. 
Moreover, the representations $m_1^l / m_1^{l+1}$ and $m_1^{l'} / m_1^{l'+1}$ are isomorphic if and only if $l+l' =2p-2$.

To prove the lemma in general, consider
the restriction of the representation of $\lsp(2n)$ on $m_n^l / m_n^{l+1}$  to the Lie subalgebra 
$$ \lsp(2)^{\oplus ^n} \mono \lsp(2n)$$ of the block diagonal matrices. The latter representation decomposes as follows 
	\begin{equation}\label{rep}
	m^l_n/m_n^{l+1}  =
	 \bigoplus_{i_1 + \ldots +  i_n = l} m_1^{i_1} / m_1^{i_1 + 1} 
	 \otimes \ldots \otimes m_1^{i_n} / m_1^{i_n + 1}. 		
 	\end{equation}
	By the Jacobson density theorem the representation of  $ \lsp(2)^{\oplus ^n} $ on each summand is irreducible. Moreover, if $l< 2(p-1)$ then these direct summands are pairwise non-isomorphic. 
	It follows that any subspace $V\subset m^l_n/m_n^{l+1}$ invariant under the $ \lsp(2)^{\oplus ^n} $-action is the sum of {\it some} of the summands appearing in (\ref{rep}). Hence, it suffices to prove that 
	if a $\lsp(2n)$-subrepresentatioin $W \subset  m^l_n/m_n^{l+1}  $ contains $m_1^{i_1} / m_1^{i_1 + 1}  \otimes \ldots \otimes m_1^{i_n} / m_1^{i_n + 1}$, for some partition $(i_1,  \ldots ,  i_n)$ of $l$ with $i_1>0$,
	the projection of $W$ to $m_1^{i_1-1} / m_1^{i_1} \otimes   m_1^{i_2+1} / m_1^{i_2+2}    \otimes \ldots \otimes m_1^{i_n} / m_1^{i_n + 1}$ is nonzero. This reduces the proof to the case $n=2$ which is shown by 
	direct inspection.  
	\end{proof}

\end{document}